\documentclass[12pt, a4paper]{amsart}

\usepackage{ifthen}
\usepackage{tikz}

\usepackage{dsfont} 

\newtheorem*{rep@theorem}{\rep@title}
\newcommand{\newreptheorem}[2]{%
\newenvironment{rep#1}[1]{%
 \def\rep@title{#2 \ref{##1}}%
 \begin{rep@theorem}}%
 {\end{rep@theorem}}}
\makeatother

\usepackage{breqn}
\usepackage{amscd,amsmath,amssymb,amsthm,amsfonts}
\usepackage[alphabetic]{amsrefs}
\usepackage{mathrsfs}
\usepackage[shortlabels]{enumitem}
\usepackage[all]{xy}

\usepackage{thmtools}

\usepackage{xcolor}
\usepackage[hypertexnames=true,colorlinks = true, allcolors=blue,linktoc = all, pdffitwindow = false, urlbordercolor = white]{hyperref}%

\usepackage{tikz}
 
\usetikzlibrary{knots}
\usetikzlibrary{decorations.markings}
\usepackage{hyperref}
\usepackage{graphicx} 

\def\End{\operatorname{End}}

\def\supp{\operatorname{supp}}

\def\dim{\operatorname{dim}}

\def\id{\operatorname{id}}

\def\Ad{\operatorname{Ad}}

\def\Im{\operatorname{Im}}

\def\id{\operatorname{id}}
\def\supp{\operatorname{supp}}

\newcommand{\IC}[0]{\mathbb{C}}

 \newcommand{\IN}[0]{\mathbb{N}}
 
 \newcommand{\IR}[0]{\mathbb{R}}
 \newcommand{\IT}[0]{\mathbb{T}}

 \newcommand{\IZ}[0]{\mathbb{Z}}

\newcommand{\CA}[0]{\mathcal{A}} \newcommand{\CB}[0]{\mathcal{B}}
\newcommand{\CC}[0]{\mathcal{C}} 
 \newcommand{\CF}[0]{\mathcal{F}}
 \newcommand{\CH}[0]{\mathcal{H}}
 
\newcommand{\CK}[0]{\mathcal{K}} 
 
\newcommand{\CO}[0]{\mathcal{O}} 
 
 \newcommand{\CT}[0]{\mathcal{T}}

\newcommand{\TL}[2]{%
\vcenter{\hbox{\begin{tikzpicture}
\foreach \x/\y/\z/\w in {#2} {
  \ifthenelse{\x = \z \AND \y = \w}
  {}
  {
    \ifthenelse{\y = \w}
    {\ifthenelse{\y = 0}
      {\draw (\x, \y) .. controls +(0, 0.5) and +(0, 0.5) .. (\z, \w);}
      {\draw (\x, \y) .. controls +(0, -0.5) and +(0, -0.5) .. (\z, \w);}
    }
    {\draw (\x, \y) .. controls +(0, 1) and +(0, -1) .. (\z, \w);}
  }
}
\foreach \x/\y/\z/\w in {#2} {
  \fill (\x, \y) circle (2pt); 
  \fill (\z, \w) circle (2pt); 
}
\draw (-0.5, 0) rectangle ({#1 - 0.5}, 1);
\clip (-0.6, 0) rectangle ({#1 - 0.6}, 1);
\end{tikzpicture}
}}}

\newreptheorem{theorem}{Theorem}
\newtheorem{theorem}{Theorem}[section]
\newtheorem*{theorem*}{Theorem}
\newtheorem*{proposition*}{Proposition}
\newtheorem{proposition}[theorem]{Proposition}
\newtheorem{lemma}[theorem]{Lemma}
\newtheorem*{lemma*}{Lemma}
\newtheorem{example}[theorem]{Example}

\newtheorem{definition}[theorem]{Definition}

\newtheorem{corollary}[theorem]{Corollary}
\newtheorem{remark}[theorem]{Remark}

\numberwithin{equation}{section}

\begin{document}

\allowdisplaybreaks  

 \title{The Motzkin subproduct system}
\author{Valeriano Aiello} 
\address{Valeriano Aiello, Dipartimento di Matematica, Universit\`a di Roma Sapienza, P.le Aldo Moro 5, 00185 Roma, Italy, \url{https://github.com/valerianoaiello}}
\email{valerianoaiello@gmail.com}
\author{Simone Del Vecchio}
 \address{Simone Del Vecchio, Dipartimento di Matematica, Universit\`a degli studi di Bari, Via E. Orabona, 4, 70125 Bari, Italy}
\email{simone.delvecchio@uniba.it}
\author{Stefano Rossi}
 \address{Stefano Rossi, Dipartimento di Matematica, Universit\`a degli studi di Bari, Via E. Orabona, 4, 70125 Bari, Italy}
\email{stefano.rossi@uniba.it}

\begin{abstract}
We introduce a subproduct system of finite-dimensional Hilbert spaces by using the Motzkin planar algebra and its Motzkin Jones-Wenzl idempotents,  which generalizes the Temperley-Lieb subproduct system of Habbestad and Neshveyev.
 We provide a description of the
 corresponding Toeplitz  and Cuntz-Pimsner 
C$^*$-algebras as universal 
C$^*$-algebras, defined in terms of generators and relations, and we highlight properties of  their representation theory. 
\end{abstract}

\maketitle

\tableofcontents

\section*{Introduction}
Subproduct systems were introduced independently 
in  \cite{Shalit}  and
  \cite{Bhat}, with the latter referring to them as inclusion systems. In the context relevant to this paper, subproduct systems consist of a family of finite-dimensional Hilbert spaces (satisfying certain axioms) that can be combined together to form a type of Fock space, on which the generators of a Toeplitz $C^*$-algebra act. 
Although the initial source of motivation for defining subproduct systems was dilation theory for semigroups of completely positive maps,
since their introduction, subproduct systems have garnered significant interest, leading to the construction of remarkable 
$C^*$-algebras. 
In the spirit of the programme of \cite{Popescu, Popescu2}, their study can be viewed as a form of noncommutative algebraic geometry, as there is a one-to-one correspondence between subproduct systems and homogeneous ideals in  the algebra of non-commutative polynomials $\IC\langle X_1, \ldots, X_n\rangle$, a sort of  non-commutative Nullstellensatz.

In the present paper we want to make use of the Motzkin algebra to construct a novel subproduct system.
The Motzkin algebra was initially introduced by Benkart and Halverson in \cite{Halverson}. 
Later, Hatch, Ly, and Posner \cite{Ly}   explored its structure, unveiling its presentation, while Jones and Yang \cite{Jo21} leveraged this algebra to develop fusion categories governed by $A_n$ fusion rules.

  Motzkin algebras are a family of algebras $M_k(\lambda^{-1})$, where $k$ is any non-negative integer and $\lambda$ is a real parameter.
Similarly to the Temperley-Lieb algebras,   elements of   $M_k(\lambda^{-1})$
can be described as linear combinations of  rectangles, with $k$ points on the top side of the rectangle, $k$ on the bottom side. 
Inside each rectangle, there are up to $k$ edges joining the $2k$ points on the sides, without creating crossings.

Our interest in this  algebra stems from Habbestad and Neshveyev's work on subproduct systems \cite{Nesh, Nesh2}, where they introduce and investigate a C$^*$-algebra defined by means of the Temperley-Lieb algebra and of the Jones-Wenzl idempotents.
The corresponding subproduct systems are the so-called Temperley-Lieb sub-product systems, whose corresponding ideal
is generated by a single homogeneous polynomial of degree $2$. 
The study of the simplest of these polynomials goes back to Arveson \cite{Arveson98}, who first studied the subproduct system corresponding to
the ideal $\langle X_1X_2-X_2X_1\rangle$, whose Hilbert spaces afford an action of the unitary group $U(2)$ and make the subproduct system 
$U(2)$-equivariant.
A couple of decades later, in \cite{AK}  a family of  $SU(2)$-equivariant subproduct systems
 were defined and thoroughly analyzed. 
In \cite{Nesh, Nesh2}, some of the results of the foregoing paper were finally
strengthened and generalized exhibiting the first subproduct system endowed with a quantum group symmetry.
 Cuntz-Pimsner algebras of Temperley-Lieb sub-product systems were also studied in \cite{Nesh3} by means of $KK$-theory, where their $K$-homology were explicitly described.
Given that the Motzkin algebra generalizes the Temperley-Lieb algebra - encompassing similar idempotents - we aim to extend their analytical framework 
of the latter algebra
to the Motzkin algebra. Notably, both algebras belong to the broader class of planar algebras \cite{jo2}, whose definition was originally motivated by the study of subfactors but has also proven to be of interest beyond that domain, as for instance they provide representations of braid groups and Thompson groups, see e.g. \cite{Jonespolynomial, Jo14, Jo19, AJ}.

The cornerstone of the construction pursued in \cite{Nesh} was the notion of a Temperley-Lieb vector, a vector in  $H\otimes H$ that enabled the definition of a representation of the Temperley-Lieb algebra. In this article, we introduce the concept of a Motzkin pair, consisting of a Temperley-Lieb vector and another vector  $v\in H$. 
The introduction of
 the second vector 
 allows us to widen the class  
 of ideals corresponding to the Temperley-Lieb subproduct system. In \cite{KS}, a thorough analysis was conducted on ideals generated by monomials. Like the ideals studied in \cite{AK, Nesh, Nesh2}, those examined in the present paper are homogeneous and generated by a single degree-2 polynomial  $P$.
In addition, our ideals also include the subproduct systems studied in \cite{Nesh} by choosing a suitable Motzkin pair.
To each representation of the Motzkin algebra, it is possible to associate two $C^*$-algebras: the Toeplitz $\CT_P$ 
and the Cuntz-Pimnser C$^*$-algebras $\CO_P$. The former always contains the compact operators as an ideal, while the latter is the quotient with respect to that ideal. Both C$^*$-algebras admit an action of $\IT$, by mapping each generating isometry to a scalar multiple of itself and this is a key tool to providing a characterization of $\CT_P$ as a universal C$^*$-algebra. 
 
In Section \ref{sec1}, we revisit the definition and foundational properties of Motzkin algebras, introducing the concept of a Motzkin pair, which eases the construction of representations. We present three families of representations. 
Combining the first two yields the same subproduct
 as that studied in \cite{Nesh}. The third is novel.
 All these families of representations are   shown to be faithful for suitable values of the parameter $\lambda$ of the Motzkin algebra. Section \ref{sec2} elaborates on Motzkin Jones-Wenzl idempotents in Motzkin algebras and standard subproduct systems, culminating in the construction of Motzkin subproduct systems and the determination of their corresponding ideals. Finally, Section \ref{sec3} is focused on the C$^*$-algebras constructed from the  family of representations of the Motzkin algebra that we exhibited and presents a key result: the characterization of its Toeplitz 
 and Cuntz-Pimsner algebras  
 as universal C$^*$-algebras. 
 
\section{The Motzkin algebra}\label{sec1}
In this section, after recalling the definition and some key properties of the Motzkin algebra, we define some representations 
of it. For more detailed information, refer to \cite{Halverson, Ly, Jo21}.

For $k=1$,  $M_1(\lambda^{-1})$ is the unital universal $*$-algebra generated by a self-adjoint idempotent $p_1$. This algebra is isomorphic to $\IC\oplus \IC$. 
For $k\geq 2$,  the Motzkin $*$-algebra $M_k(\lambda^{-1})$ is the algebra generated by $1$, $t_1$, \ldots, $t_{k-1}$, $l_1$, \ldots, $l_{k-1}$, satisfying the following relations \cite[Theorem 4.1]{Ly}
\begin{enumerate}[start = 0] \label{relations-motzkin}
\item $t_i=t_i^*$
\item $l_i^2=l_i^3$
\item $l_i l_{i+1} l_i = l_i l_{i+1} = l_{i+1} l_i l_{i+1}$
\item $l_{i} l_{i}^* l_{i} = l_{i}$
\item $l_{i+1} l_{i}^* l_{i} = l_{i+1} l_{i}^* $, $l_{i} l_{i}^* l_{i-1} = l_{i}^* l_{i-1}$
\item $l_{i} l_i^* = l_{i+1}^* l_{i+1}$
\item if $|i-j|\geq 2$, one has $l_{i}^* l_{j} = l_j l_i^*$, $l_il_j=l_jl_i$,  $t_it_j=t_jt_i$,
$l_i t_j=t_jl_i$, $l_i^* t_j=t_jl_i^*$, 
\item $t_i^2=t_i$
\item $t_it_{i+1}t_i=\lambda^{2} t_i$, $t_{i+1}t_{i}t_{i+1}=\lambda^{2} t_{i+1}$, 
\item $t_il_i = t_i l_i^*$
\item $\lambda t_i l_{i+1}^* = t_i t_{i+1} l_i$
\item $l_i^* l_{i+1}^* t_i = t_{i+1} l_i^* l _{i+1}^*$
\item $t_i l_i t_i = \lambda t_i$
\end{enumerate}
It can be proved that   $\dim M_k$ is the $(2k)$-th Motzkin number, \cite{Halverson}.
 The elements $1$, $t_1$, \ldots, $t_{k-1}$, along with the relations in (0), (6), (7), (8), provide a presentation of the Temperley-Lieb algebra TL$_k(\lambda^{-1})$. 
 In the sequel we will also make use of the elements  $p_i:=l_i^*l_i$, for $1\leq i\leq k-1$, $p_{k}:=l_{k-1}l_{k-1}^*$.
  
The Motzkin algebra $M_k(\lambda^{-1})$ has a convenient graphical representation that is particularly useful for proving algebraic relations within the algebra. The basis of
 $M_k(\lambda^{-1})$ consists of diagrams in boxes with $k$ points along the top edge and $k$ points along the bottom edge, where each point is connected to at most one other point by a non-crossing planar edge within the box. Specifically, there can be at most $k$ edges inside the box.
The diagrams containing exactly $k$ edges form the basis of $TL_k(\lambda^{-1})$. The former elements are referred to as Motzkin diagrams, the latter are the Temperley-Lieb diagrams.

To multiply  two basis elements $d_1$ and $d_2$, 
 stack $d_1$ on top of $d_2$ forming a new box called $d_3$. Identify the bottom-row vertices in $d_1$ with the corresponding top-row vertices in $d_2$.  
Replace any loops that may have formed with the parameter  $\lambda^{-1}$; 
 and remove the edges inside $d_3$
  having at least an endpoint  on the bottom edge of   $d_1$. 
Here is an example,
\[
  d_1 \cdot d_2 =  \TL{3}{0/0/0/0, 1/0/2/0, 0/1/1/1, 0/1/1/1, 2/1/2/1}
  \cdot \TL{3}{0/0/1/0, 2/0/0/1, 1/1/2/1}
  = \lambda^{-1} \cdot \TL{3}{0/0/1/0, 0/1/1/1, 2/0/2/0, 2/1/2/1}
\]
There is an (anti-linear) involution $*: M_k(\lambda^{-1})\to M_k(\lambda^{-1})$ defined at the level of Motzkin diagram as reflection about a horizontal line.

There is a natural inclusion map   $\iota: M_k(\lambda^{-1})\to M_{k+1}(\lambda^{-1})$, which works as follows: expand the box slightly to the right, adding a point on the bottom edge at the right end and a corresponding point on the top edge at the right end. Then, draw a vertical edge connecting these two new points.
Here is an example
\[
\iota\left(  \TL{3}{0/0/0/0, 1/0/2/0, 0/1/1/1, 0/1/1/1, 2/1/2/1}\right) =   \TL{4}{0/0/0/0, 1/0/2/0, 0/1/1/1, 0/1/1/1, 2/1/2/1, 3/0/3/1}
\]
To ease the notation, we will often omit the symbol $\iota$, however, when two elements  of different Motzkin algebras $x\in M_k(\lambda^{-1})$ and $y\in M_{k+h}(\lambda^{-1})$  appear in the same expression, $x$ should be understood as $\iota^h(x)$.

In this description, the element $l_i$ is a box with vertical edges between all points on the top and bottom edges, except between the $i$-th and
$(i+1)$-th where there is only an edge between the $i$-th on the top and the $(i+1)$-th on the bottom. 
The identity consists of the box with $k$ vertical edges. Sometimes we will denote by $r_i$ the adjoint of $l_i$. See Figure \ref{gen_Mk}.
The tensor product of the algebras $M_a(\lambda^{-1})$ and $M_b(\lambda^{-1})$ 
is naturally injected into $M_{a+b}(\lambda^{-1})$,
 namely by the horizontal juxtaposition 
on the level of boxes. In the sequel, we will make use of this identification.

\begin{figure}
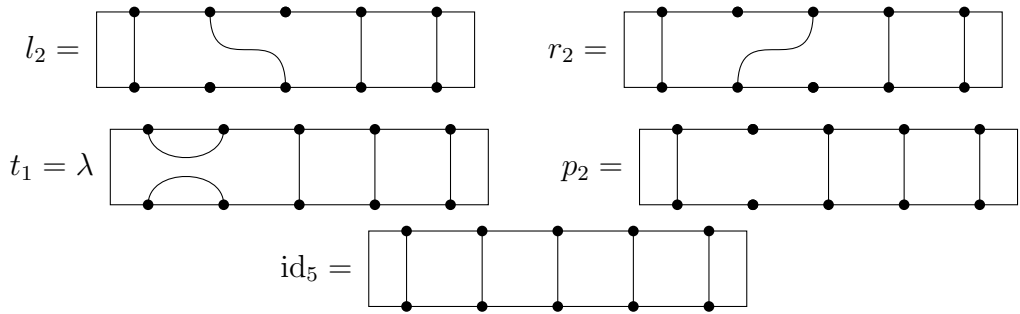

\phantom{This text will be invisible} 
\[
l_2= \TL{5}{0/0/0/1, 2/0/1/1, 3/0/3/1, 4/0/4/1, 1/0/1/0, 2/1/2/1}
\qquad r_2= \TL{5}{0/0/0/1, 1/0/2/1, 3/0/3/1, 4/0/4/1, 2/0/2/0, 1/1/1/1}
\]

\[
 t_1 = \lambda\; \TL{5}{0/0/1/0, 0/1/1/1, 2/0/2/1, 3/0/3/1, 4/0/4/1}\qquad
p_2= \TL{5}{0/0/0/1, 1/0/1/0, 1/1/1/1, 2/0/2/1, 3/0/3/1, 4/0/4/1}
\]
\[
\id_5= \TL{5}{0/0/0/1, 1/0/1/1, 2/0/2/1, 3/0/3/1, 4/0/4/1}
\]
\caption{The elements $l_2$, $r_2=l_2^*$, $t_1$, $p_2$, and the identity $\id_5$ of $M_5(\lambda^{-1})$.}\label{gen_Mk}
\end{figure}

We now recall 
the notion of Temperley-Lieb vector due to Habbestad and Neshveyev.
\begin{definition}\cite[Definition 2.1]{Nesh}
Given a Hilbert space $H$, with $2\leq \dim H<\infty$,  a vector $w\in H\otimes H$ is said to be  Temperley-Lieb if the projection $t=[\IC w]\in B(H\otimes H)$ satisfies 
\begin{align}
& (t\otimes 1) (1\otimes t) (t\otimes 1) = \lambda^{2} (t\otimes 1) \text{ in } B(H\otimes H\otimes H)\, .
\end{align} 
\end{definition}
Each Temperley-Lieb vector yields a representation of the Temperley-Lieb algebra with parameter $\lambda^{-1}$ that we will describe in detail later, see \cite[Section 1]{Nesh} and Proposition \ref{extended-rep}.
From now on, we assume that Temperley-Lieb vectors have norm $1$. 

In the rest of the paper we will make use of the following notation as it allows us to shorten many formulas.
For $i\in\{1, \ldots , n\}$, we set
\begin{align}
\bar i:=n-i+1 
\end{align}
By  \cite[Lemma 1.4, Prop. 1.5]{Nesh}, each Temperley-Lieb vector can be written as $v_A=\sum_{i=1}^n v_i\otimes Av_i$, where $v_i$ is an orthonormal basis of $H=\IC^n$, 
$A\in M_{n}(\IC)$ is such that $Av_i=a_i v_{\bar i}$ 
and $A^2$ is a unitary up to a scalar, i.e. $(A^2)^*A^2=\alpha \cdot\id$ for some $\alpha>0$. 

Throughout this paper, the inner product of a Hilbert space will be linear in the first variable and  anti-linear in the second.

We recall the following result that will come in handy in the sequel, cf. \cite[proof of Prop. 1.7]{Nesh}. 
\begin{lemma} \label{formula_E}
For any Temperley-Lieb vector $v_A$, one has
\begin{align}
 t&= \sum_{h,k=1}^n  \bar{a}_{  h} a_k (e_{kh}\otimes e_{\bar k \bar h})\,,
 \end{align}
 where $\{e_{kh}\}_{h,k=1}^n$ are the matrix units. 
 \end{lemma}
\begin{proof} 
We have
 \begin{align*}
 &t(v_i\otimes v_j)=\langle v_i\otimes v_j, v_A\rangle v_A
 =\langle v_i\otimes v_j, \sum_{k=1}^n  a_k v_k\otimes v_{\bar k}\rangle v_A\\
& = \sum_{k=1}^n \langle v_i\otimes v_j, a_k v_k\otimes v_{\bar k}\rangle v_A
= \sum_{k=1}^n \delta_{i, k} \delta_{j, \bar k} \bar{a}_k  v_A
 =  \delta_{i,  \bar j} \bar{a}_i  v_A
 = \sum_{k=1}^n \delta_{i,  \bar j}   \bar{a}_i  a_k v_k \otimes v_{\bar k}\\
&= \sum_{k=1}^n  \bar{a}_{  i} a_k  \delta_{i, \bar j} (v_k\otimes v_{\bar k})
= \sum_{h,k=1}^n  \bar{a}_{  h} a_k \delta_{i,h}	\delta_{\bar h, j}	(e_{kh}\otimes e_{\bar k \bar h}) (v_i\otimes v_j)\\
&= \sum_{h,k=1}^n  \bar{a}_{  h} a_k (e_{kh}\otimes e_{\bar k \bar h}) (v_i\otimes v_j)\,,
 \end{align*}
 where in the last equality we used that $e_{bc}v_d=\delta_{c,d}v_b$.
 \end{proof} 
As we said, one of our main goals is  to construct representations of the Motzkin algebras.

\begin{proposition} \label{extended-rep} 
With the notation as above, 
let $v_A\in H\otimes H$ be a Temperley-Lieb vector and
$v=\sum_{i=1}^n b_i v_i \in H$ a norm-$1$ 
vector 
 satisfying the following conditions
\begin{align*}
&\lambda = |a_i a_{\bar i}|\, , \qquad
 \sum_{k=1}^n |a_k|^2=1\, , \qquad
\sum_{h=1}^n |b_h|^2 =1\, , \\
&  \overline{a }_j \bar{b}_i  b_{\bar j} =     \bar{a}_{\bar i} \bar{b}_j b_{\bar i}\, ,  
 \qquad
    \lambda \bar{a}_i \bar{b}_k  b_{\bar i}= \bar{a}_{j} a_{\bar j} \bar{a}_{\bar k}  \bar{b}_i b_{\bar k}    \,   ,     
\qquad
  \sum_{k=1}^{n}  \bar{a}_{\bar k}  a_k |b_k|^2=\lambda \, ,
 \end{align*}
 for all $i$, $j$, $k$. 
 Set  
 \begin{equation*}
\begin{aligned}
&p(x)= \langle x,v\rangle v \\
&t(x\otimes y)= \langle x\otimes y, v_A \rangle v_A\\
&l(v_i\otimes v_j)= v_j\otimes p(v_{i})\\
&l^*(v_i\otimes v_j)=  p(v_{j})\otimes v_i
\end{aligned} 
\end{equation*}
Then there exists  a representation of the Motzkin algebra $M_r(\lambda^{-1})$ on $H^{\otimes r}$ defined as
\begin{equation*}
\begin{aligned}
&t_i = \id^{\otimes (i-1)} \otimes t \otimes \id^{\otimes (r-i-1)} \\
&l_i = \id^{\otimes (i-1)} \otimes l \otimes \id^{\otimes (r-i-1)} \qquad 1 \leq i \leq r-1 \\
&p_i = \id^{\otimes (i-1)} \otimes p \otimes \id^{\otimes (r-i)} \qquad 1 \leq i \leq r \\
\end{aligned}
\end{equation*}
where $\id\in\CB(H)$.
\end{proposition}
\begin{proof} 
We need to verify that conditions (0)–(12) are satisfied. Conditions (0), (7), and (8) establish that
 $v_A$ is a Temperley-Lieb vector. 

Condition (1) reads as
\begin{align*}
l^2(v_i\otimes v_j)&=p(v_i)\otimes p(v_j)
=p(v_j)\otimes p(v_i)
=l^3(v_i\otimes v_j)
\end{align*}
which holds for any $v$ as $p(v_i)=\bar{b}_iv$.
 
Condition (2) reads as
\begin{align*}
&(l\otimes \id)(\id \otimes l)(l\otimes \id)(v_i\otimes v_j\otimes v_k)=v_k\otimes p(v_j)\otimes p(v_i)\\
&=v_k\otimes p(v_i)\otimes p(v_j)=(l\otimes \id)(\id\otimes l)(v_i\otimes v_j\otimes v_k)\,,
\end{align*}
which holds for any $v$.

Condition (3)  expresses the fact that $l$ is a partial isometry. In order for $l$ to be such an operator we need 
\begin{align*}
&ll^*l(v_i\otimes v_j)=ll^*(v_j \otimes p(v_i))=l(p(p(v_i)) \otimes v_j)\\
&= v_j \otimes p(p(p(v_i)))
= v_j\otimes p(v_i)=l(v_i\otimes v_j).
\end{align*}
 
As for the first of (4), 
 we have
\begin{align*}
&(\id \otimes l)(l^*\otimes \id)(l\otimes \id)(v_i\otimes v_j \otimes v_k)=(\id \otimes l)( l^*\otimes \id)(  v_j\otimes p(v_{i})\otimes v_k)\\
&=(\id \otimes l)( p(p(v_{i}))\otimes v_j\otimes v_k)=p(p(v_{i}))\otimes v_k\otimes p(v_{j})\\
&=p(v_{i})\otimes v_k\otimes p(v_{j})=p(v_j)\otimes v_k\otimes p(v_i)\\
&=(\id \otimes l)(p(v_j)\otimes v_i\otimes v_k)=(\id \otimes l)(l^*\otimes \id)(v_i\otimes v_j \otimes v_k)
\end{align*}
 Likewise, for the second equation we have
\begin{align*}
&(\id \otimes l)(\id \otimes l^*)(l\otimes \id)(v_i\otimes v_j \otimes v_k)=(\id \otimes l)(\id \otimes l^*)(v_j\otimes p(v_i) \otimes v_k)\\
&=(\id \otimes l) (v_j\otimes p(v_k) \otimes p(v_i))=  v_j\otimes p(v_i) \otimes p(p(v_k))\\
&=  v_j\otimes p(v_k) \otimes p(v_i)=(\id \otimes l^*)(v_j\otimes p(v_i) \otimes v_k)\\
&=(\id \otimes l^*)(l \otimes \id)(v_i\otimes v_j \otimes v_k)\,.
\end{align*}

As for (5)
\begin{align*}
&(\id \otimes l^*)(\id\otimes l)(v_i\otimes v_j \otimes v_k)=(\id \otimes l^*)( v_i\otimes v_k\otimes p(v_j))\\
&=   v_i\otimes p(p(v_j))\otimes v_k=(l \otimes \id)(  p(v_{j})\otimes v_i\otimes v_k)\\
&=(l \otimes \id)(l^*\otimes \id)(v_i\otimes v_j \otimes v_k)\,.
\end{align*}

(6) is obvious.

Condition (9) imposes the first restriction on $v$, namely 
 $$
  \overline{a }_j \bar{b}_i  b_{\bar j} =     \bar{a}_{\bar i} \bar{b}_j b_{\bar i}\,.
$$
Indeed, we have
\begin{align*}
&tl(v_i\otimes v_j)=t(v_j\otimes p(v_i))= \langle  v_j\otimes p(v_i),  \sum_{k=1}^{n} a_k v_k\otimes v_{\bar k}\rangle  v_A\\
&= \sum_{k=1}^{n}  \langle  v_j, v_k \rangle \langle p(v_i), a_k  v_{\bar k}\rangle  v_A
= \sum_{k=1}^{n}  \langle  v_j, v_k \rangle \langle \bar{b}_iv, a_k  v_{\bar k}\rangle  v_A\\
&= \sum_{k=1}^{n}  \delta_{j,k}  \langle \bar{b}_iv, a_k  v_{\bar k}\rangle  v_A
= \sum_{k=1}^{n}    \delta_{j,k}    \bar{b}_i b_{\bar k} \bar{a}_k     v_A=    \overline{a }_j \bar{b}_i  b_{\bar j}   v_A\\
&=     \bar{a}_{\bar i} \bar{b}_j b_{\bar i}   v_A=   \sum_{k=1}^{n}    \bar{b}_j b_k  \bar{a}_k   \delta_{i, \bar k}  v_A
 =   \sum_{k=1}^{n}  \langle  \bar{b}_j v,  a_k v_k \rangle \langle v_i,  v_{\bar k}\rangle  v_A\\
&= \langle p(v_j)\otimes v_i,  \sum_{k=1}^{n} a_k v_k\otimes v_{\bar k}\rangle  v_A=t( p(v_j)\otimes v_i)
=tl^*(v_i\otimes v_j)\,.
\end{align*}

Condition (10) requires 
$$
    \bar{a}_{j} a_{\bar j} \bar{a}_{\bar k}  \bar{b}_i b_{\bar k}           = \lambda \bar{a}_i \bar{b}_k  b_{\bar i}
$$ 
to hold. Indeed, 
 one has
\begin{align*}
&\lambda (t\otimes \id)(\id \otimes l^*)(v_i\otimes v_j \otimes v_k) = \lambda(t\otimes \id)(  v_i\otimes p(v_k) \otimes v_j) \\
 &= \lambda \langle v_i\otimes p(v_k),  \sum_{r=1}^{n}   a_r \, v_r\otimes v_{\bar r} \rangle v_A\otimes v_j  \\
 &
 = \lambda \langle v_i\otimes \bar{b}_k  v,  \sum_{r=1}^{n}   a_r \, v_r\otimes v_{\bar r} \rangle v_A\otimes v_j  = \lambda\sum_{r=1}^{n} \delta_{i,r} \bar{a}_r \bar{b}_k    \langle   v,   v_{\bar r} \rangle v_A\otimes v_j  \\
 &
 = \lambda \bar{a}_i \bar{b}_k  b_{\bar i}    v_A\otimes v_j\,,
 \end{align*}
and
\begin{align*}
&  (t\otimes \id) (\id\otimes t) (l \otimes \id)(v_i\otimes v_j \otimes v_k) =(t\otimes \id) (1\otimes t) ( v_j\otimes p(v_i) \otimes v_k)\\
&= (t\otimes \id) \left( \langle   \bar{b}_i v \otimes v_k,  \sum_{r=1}^{n} a_r  v_r\otimes v_{\bar r}\rangle\, v_j\otimes v_A\right)  \\
&=   \sum_{r=1}^{n} (t\otimes \id)  \left( \bar{a}_r \bar{b}_i  b_r \delta_{k,\bar r} \, v_j\otimes v_A\right)  
=   (t\otimes \id)  \left( \bar{a}_{\bar k} \bar{b}_i b_{\bar k}  \, v_j\otimes v_A\right)  \\
&=  (t\otimes \id)  \left(  \sum_{r=1}^{n} a_r \bar{a}_{\bar k}  \bar{b}_i b_{\bar k} \, v_j\otimes v_r\otimes v_{\bar r}\right)  
=    \sum_{r=1}^{n} a_r \bar{a}_{\bar k}  \bar{b}_i b_{\bar k} \, (t\otimes \id) (v_j\otimes v_r\otimes v_{\bar r})  \\
&=   \sum_{r=1}^{n} a_r \bar{a}_{\bar k}  \bar{b}_i b_{\bar k}  \,  \langle    v_j\otimes v_r,  \sum_{r'=1}^{n} a_{r'} v_{r'}\otimes v_{\bar r'} \rangle v_A\otimes v_{\bar r}    \\
 &
=   \sum_{r, r'=1}^{n} \delta_{j,r'}  \bar{a}_{r'} a_r \bar{a}_{\bar k}  \bar{b}_i b_{\bar k}      \delta_{r,\bar r'}  \,   v_A\otimes v_{\bar r} 
=     \bar{a}_{j} a_{\bar j} \bar{a}_{\bar k}  \bar{b}_i b_{\bar k}           \,   v_A\otimes v_{j}\,.
 \end{align*}

As these computations show, (11) is automatically satisfied
\begin{align*}
&  (l^*\otimes \id)(\id \otimes l^*)(t\otimes \id)(v_i\otimes v_j \otimes v_k)\\
&= \sum_{r=1}^{n} \bar{a}_r (l^*\otimes \id)(\id \otimes l^*) (\langle v_i\otimes v_j, v_r \otimes v_{\bar r}\rangle v_A\otimes v_k)\\
&= \sum_{r=1}^{n} \bar{a}_r \delta_{i,r} \delta_{j, \bar r} (l^*\otimes \id)(\id \otimes l^*)   (v_A\otimes v_k)\\
&=  \bar{a}_i   \delta_{j, \bar i} (l^*\otimes \id)(\id \otimes l^*)   (v_A\otimes v_k)\\
&=   \sum_{r=1}^{n} a_r \bar{a}_i   \delta_{j, \bar i} (l^*\otimes \id)(\id \otimes l^*)   (v_r\otimes v_{\bar r}\otimes v_k)\\
&=   \sum_{r=1}^{n} a_r \bar{a}_i   \delta_{j, \bar i} (l^*\otimes \id)  (v_r\otimes p(v_k)\otimes v_{\bar r})\\
&=   \sum_{r=1}^{n} a_r \bar{a}_i   \delta_{j, \bar i} (p(v_k)\otimes v_r\otimes v_{\bar r})
=   \bar{a}_i   \delta_{j, \bar i} \, \bar{b}_kv\otimes v_A\\
&=  \sum_{r=1}^{n} \delta_{r,i} \delta_{j,\bar r}\bar{a}_r     \bar{b}_kv\otimes v_A
=  \sum_{r=1}^{n}   \langle v_i \otimes v_j, a_r v_r\otimes v_{\bar r} \rangle \bar{b}_kv\otimes v_A\\
&=   (\id\otimes t)( \bar{b}_kv\otimes v_i \otimes v_j)
=  (\id\otimes t)( p(v_k)\otimes v_i \otimes v_j)\\
&=  (\id\otimes t)(l^* \otimes \id)( v_i\otimes p(v_k) \otimes v_j)
=  (\id\otimes t)(l^* \otimes \id)(\id\otimes l^*)(v_i\otimes v_j \otimes v_k)\,.
\end{align*}

As for (12), one has
 \begin{align*}
&tlt(v_i\otimes v_j)=tl(\langle v_i\otimes v_j,  \sum_{k=1}^{n} a_k v_k\otimes v_{\bar k}\rangle v_A)  \\
&=tl(  \sum_{k=1}^{n} \delta_{i,k} \delta_{j,\bar k} \bar{a}_k v_A)   =tl(  \delta_{j,\bar i} \bar{a}_i v_A )
=tl(  \sum_{k=1}^{n}  \delta_{j,\bar i} \bar{a}_i a_k \; v_k\otimes v_{\bar k}) \\
&=t(  \sum_{k=1}^{n}   \delta_{j,\bar i} \bar{a}_i a_k \; v_{\bar k}\otimes p(v_k)) 
=t(  \sum_{k=1}^{n}  \delta_{j,\bar i} \bar{a}_i a_k \bar{b}_k\; v_{\bar k}\otimes v) \\
&=\sum_{k=1}^{n}  \delta_{j,\bar i} \bar{a}_i a_k \bar{b}_k\; t(   v_{\bar k}\otimes v) 
=\sum_{k, k'=1}^{n}  \delta_{j,\bar i} \bar{a}_i a_k \bar{b}_k\; \langle   v_{\bar k}\otimes v, a_{k'} v_{k'}\otimes v_{\bar k'}\rangle v_A  \\
&=\sum_{k, k'=1}^{n}  \delta_{j,\bar i} \delta_{\bar k, k'} \bar{a}_{k'} \bar{a}_i a_k |b_k|^2\, v_A  
=\sum_{k=1}^{n}  \delta_{j,\bar i}  \bar{a}_{\bar k} \bar{a}_i a_k |b_k|^2\, v_A\,, 
\end{align*}
while 
 \begin{align*}
\lambda t(v_i\otimes v_j)&= \lambda  \langle v_i\otimes v_j,  \sum_{k=1}^{n} a_k v_k\otimes v_{\bar k}\rangle v_A  
=  \sum_{k=1}^{n} \lambda \delta_{i,k} \delta_{j,\bar k} \bar{a}_k v_A  = \lambda \delta_{j,\bar i} \bar{a}_i v_A\,.   
\end{align*}
Therefore we need
$$
\sum_{k=1}^{n}  \bar{a}_{\bar k}  a_k |b_k|^2=\lambda\,.
$$
\end{proof} 
We now present a definition that we will use throughout the paper.
\begin{definition}\label{defsupport}
Given a vector $v$, we call it support $\supp(v)$ the set of $i$ such that $\langle v_i, v\rangle \neq 0$.
\end{definition}
In the next lemma, we eliminate some of the conditions required in Proposition  \ref{extended-rep}.
\begin{lemma}\label{basicpropMpair}
Let $(v_A,v)$ be a pair satisfying the same hypotheses of Proposition \ref{extended-rep}. Then, we have  that
\begin{enumerate}
\item $|b_j|=|b_{\bar j}|$; in particular, $j\in \supp(v)$ if and only if $\bar j\in \supp(v)$;
\item $a_j=a_{\bar j}$ for all $j\in \supp(v)$;
\item $\lambda = \bar a_i a_{\bar i}\in\IR_+$ for all $i$. 
 \end{enumerate}
Conversely, for any pair of vectors $(v_A=\sum_{i=1}^n a_i v_i\otimes v_{\bar{i}}, v=\sum_{i=1}^n b_i v_i)$, where $v_A$ is a Temperley-Lieb vector in $H\otimes H$ and $v$ is norm-$1$ vector in $H$, 
satisfying
\begin{equation*}
\begin{aligned}
&\lambda = \bar{a}_i a_{\bar i}\in\IR_+\, , \qquad
\sum_h |b_h|^2=1\, , \qquad
  \overline{a }_j \bar{b}_i  b_{\bar j} =     \bar{a}_{\bar i} \bar{b}_j b_{\bar i}  
\end{aligned}
\end{equation*}
 the following conditions are automatically satisfied
\begin{enumerate}
\item[(4)] the condition $\sum_{k=1}^{n}  \bar{a}_{\bar k}  a_k |b_k|^2=\lambda$;
\item[(5)] the condition $ \bar{a}_{j} a_{\bar j} \bar{a}_{\bar k}  \bar{b}_i b_{\bar k}           = \lambda \bar{a}_i \bar{b}_k  b_{\bar i}$;
\end{enumerate}
\end{lemma}
\begin{proof}
1) By the fourth equation displayed in the statement of Proposition \ref{extended-rep} with $i=\bar j$, we get 
$\bar a_{ j} \bar b_{\bar{j}}b_{\bar{j}}= \bar a_j \bar b_j b_j$. Therefore, since $A^2$ is a unitary matrix up to scalar factor, $b_j\neq 0$ if and only if $b_{\bar{j}}\neq 0$.

2) We make use again of the same equation, but with $i=j$. In this case we get
$\bar a_{  j} \bar b_{j}b_{\bar{j}}= \bar a_{\bar j} \bar b_j b_{\bar j}$. Therefore, $a_j=a_{\bar j}$.

3) Let $k$ be in $\supp(v)$. 
By the fifth equation displayed in the statement of Proposition \ref{extended-rep}, with $i=k$,  we have   
\begin{align*}
 \bar{a}_{j} a_{\bar j} \bar{a}_{\bar k}  \bar{b}_k b_{\bar k}           &= \lambda \bar{a}_k \bar{b}_k  b_{\bar k}\,
\end{align*}
which by (2) is equivalent to
\begin{align*}
 \bar{a}_{j} a_{\bar j}   &= \lambda \,.
\end{align*}

4) Thanks to (2) we know that $a_j=a_{\bar j}$ whenever $j\in\supp(v)$ and thus $\lambda=|a_j|^2$.
Therefore, we may rewrite the equation as
\begin{align*}
\sum_{k=1}^{n}  \bar{a}_{\bar k}  a_k |b_k|^2&=\sum_{k=1}^{n}  |a_k|^2 |b_k|^2=\sum_{k=1}^{n}  \lambda |b_k|^2=\lambda\,.
\end{align*}

5) As $\lambda = \bar{a}_i a_{\bar i}$, the equation is equivalent to $\bar{a}_{\bar k}  \bar{b}_i b_{\bar k}           =   \bar{a}_i \bar{b}_k  b_{\bar i}$\,.
\end{proof}

In view of the preceding results, we are now ready to give the following definition.

\begin{definition}\label{Motzkin_pair}
Let $H$ be a Hilbert space, with $\dim H=n\geq 2$ and fix an orthonormal basis $\{v_i\}_{i=1}^n$.
A Motzkin pair is given by $(v_A, v)$, where $A\in M_{n}(\IC)$ has the form $Av_i=a_i v_{\bar i}$, 
 $v_A:=\sum_i v_i\otimes Av_i\in H\otimes H$ is a (unit) Temperley-Lieb vector  
 with $\lambda = |a_i a_{\bar i}|=\bar{a}_i a_{\bar i}$ for all $i\in \{1, \ldots, n\}$, $v=\sum_{i=1}^n b_iv_i$ is a unit vector of $H$,  
 satisfying the following conditions
\begin{equation}\label{eqMotzkinPair}
\begin{aligned}
&\lambda = \bar{a}_i a_{\bar i}\in\IR_+\, , \qquad
\sum_{i=1}^n |a_h|^2= \sum_{i=1}^n |b_h|^2=1\, , \qquad
  \overline{a }_j \bar b_i  b_{\bar j} =     \bar{a}_{\bar i} \bar b_j b_{\bar i} 
\end{aligned}
\end{equation}
where $1\leq i, j\leq n$.\\
Every Motzkin pair yields a representation 
$\pi_{v_{A},v}: M_k(\lambda^{-1})\to \CB(H^{\otimes k})$ by setting
for all  $1\leq i, j\leq n$
\begin{equation*}
\begin{aligned}
&p(x)= \langle x,v\rangle v \, ,\\
&t(x\otimes y)= \langle x\otimes y, v_A \rangle v_A\, ,\\
&l(v_i\otimes v_j)= v_j\otimes p(v_{i})\, , \qquad l^*(v_i\otimes v_j)=  p(v_{j})\otimes v_i\, ,
\end{aligned} 
\end{equation*}
and for $1 \leq i \leq k-1$
\begin{equation*}
\begin{aligned}
&\pi_{v_{A},v}(t_i) = \id^{\otimes (i-1)} \otimes t \otimes \id^{\otimes (k-i-1)} \\
&\pi_{v_{A},v}(l_i) = \id^{\otimes (i-1)} \otimes l \otimes \id^{\otimes (k-i-1)}  \\
&\pi_{v_{A},v}(p_i) = \id^{\otimes (i-1)} \otimes p \otimes \id^{\otimes (k-i)}  
\end{aligned}
\end{equation*}
where $\id\in\CB(H)$.
\end{definition} 
For brevity and in order to improve readability, in the sequel we will omit the symbol $\pi_{v_{A},v}$.

We now present three examples of Motzkin pairs. The third is   the most general among
those we introduce,
and we will investigate it accurately in the rest of the paper.
\begin{example}\label{seriesofexamples}
 \begin{enumerate}
\item[i)] If $\dim H=n=2m+1$, as an application of Proposition \ref{extended-rep}, a Motzkin pair is provided by 
a matrix $A$ such that $\bar{a}_ia_{\bar i}=\lambda$ for all $i\in \{1, \ldots, n\}$, and a vector $v=v_{m+1}$.  
Note that for $n=3$, 
this representation coincides with that   
of \cite[Section 3.4]{Halverson} on $(\id-p)^{\otimes k}H^{\otimes k}$, where the $k$-th Motzkin algebra is realised as
$\End_{U_q(\mathfrak{gl}_2)}(H^{\otimes k})$.
\item[ii)] In this case we consider a vector
  $v=\sum_{i=1}^n b_i v_i$ with 
$b_i=  b_{\bar{i}}\in\IR$ for $1\leq i\leq n$ and a matrix $A$ with
 $a_i=a_{j}$ for all $i, j\in \supp(v)$.
For example, one can choose
$v= (\sqrt{2})^{-1}v_1+ (\sqrt{2})^{-1}v_n$,  and   
any $a_1$ of absolute value $\lambda^{1/2}$. 
\item[iii)] One may also choose vectors with larger supports, such as, for any $r\leq n/2$, a solution is $b_i:=(\sum_{j=1}^r \delta_{i,j}+ \delta_{i,\bar{j}})/\sqrt{2r}$ with $a_1=\ldots =a_r=a_{\bar{1}}=\ldots =a_{\bar{r}}$ and $|a_1|=\lambda^{1/2}$.
\end{enumerate}
 \end{example}  
 It is natural to wonder whether the representations associated with the Motzkin pair are faithful or not.
In the next result, we provide an answer for suitable values of the parameter $\lambda$. 

We recall the map $E=E_k: M_k(\lambda^{-1})\to M_{k-1}(\lambda^{-1})$  from \cite{Jo21},
 defined on the level of diagrams as
\begin{equation}\label{conditionalexpectation}
\begin{tikzpicture}
\node at (-.35,0.25) {$\scalebox{1}{$E\big($}$};
\draw (0, 0) rectangle (1, 0.5);
\node at (1.55,0.25) {$\scalebox{1}{$\big) =\lambda$}$};

\draw (2.3, 0) rectangle (3.3, 0.5);

\draw (3.25,0) to[out=-90,in=-90] (3.5,0);
\draw (3.25,.5) to[out=90,in=90] (3.5,.5);
\draw (3.5,0) -- (3.5,.5);

 \node at (.5,0.25) {$\scalebox{1}{$x$}$};
\node at (2.8,0.25) {$\scalebox{1}{$x$}$};

 \end{tikzpicture}
\end{equation}
We now discuss the faithfulness of our representations of the Motzkin algebras for a suitable set of values of $\lambda$. 
For the remaining values, the algebra fails to be semisimple, which means it cannot  carry a C$^*$-algebraic structure.
For $\lambda^{-1}\geq 3$, the Motzkin algebras are finite-dimensional C$^*$-algebras and
the map $E$ is positive and faithful (meaning that  non-zero positive elements are mapped to non-zero positive elements).
For the above-mentioned results see  \cite[Theorem 5.14]{Halverson} and \cite[Theorem 3.2]{Jo21}.
\begin{proposition}
For any Motzkin pair $(v_A, v)$ with parameter $0<\lambda\leq 1/3$, the representation $\pi_{v_{A},v}$ in Definition \ref{Motzkin_pair} is faithful.
\end{proposition}
\begin{proof}
 Denote by $\pi_k$ our representation of $M_k(\lambda^{-1})$.
 We will make use of several maps, namely
\begin{align*}
&E_k: M_k(\lambda^{-1})\to M_{k-1}(\lambda^{-1})\\
&\tilde E_k: \pi_k(M_k(\lambda^{-1}))\to \pi_{k-1}(M_{k-1}(\lambda^{-1}))\\
&\varphi_k:  \CB(H^{\otimes k})\to  \CB(H^{\otimes k+1})\,,
 \end{align*}
where 
$E_k$ is the positive and faithful map recalled above
  (here  $\lambda^{-1}\geq 3$),
 $\varphi_k(x):=x\otimes \pi_{1}(p_{1})$ is an injective map, and
 $$
 \tilde E_k(\pi_k(x)):=\varphi_{k-1}^{-1}(\varphi_k^{-1}(\pi_{k+1}(p_{k+1}t_{k}xt_{k}p_{k+1})))\,,
 $$
 which is the restriction of 
 $$ 
 y\mapsto \varphi_{k-1}^{-1}(\varphi_k^{-1}((\id^{\otimes k-1}\otimes p\otimes p)(\id^{\otimes k-1}\otimes t) (y\otimes \id)(\id^{\otimes k-1}\otimes t)(\id^{\otimes k-1}\otimes p\otimes p)).
 $$ 
from $ \CB(H^{\otimes k})$ to $ \CB(H^{\otimes k-1})$.
  The above map
  is clearly contractive and thus a conditional expectation.
 By definition we have $\pi_k\circ E_{k+1}=\tilde E_{k+1}\circ \pi_{k+1}$.
 We give a proof of the faithfulness of $\pi_k$ by induction on $k$.
 This is trivially true for $k=1$. Suppose that it is true for $k$ and let us prove it for $k+1\geq 2$.
 It suffices to show that any non-zero $x\geq 0$ in $M_{k+1}(\lambda^{-1})$ is mapped by $\pi_{k+1}$ to something different from zero.
 Now $E_{k+1}(x)\geq 0$ and different from zero. 
 $\pi_k(E_{k+1}(x))\neq 0$ by inductive hypothesis, but this is also equal to 
 $\tilde E_{k+1}(\pi_{k+1}(x))\neq 0$. In particular, $\pi_{k+1}(x)\neq 0$.
\end{proof}

 \section{The Motzkin subproduct system} \label{sec2}
 This section is devoted to introducing the main object of this paper: the Motzkin subproduct system. In order to do this, we start by recalling the notion of standard 
subproduct system    \cite[Section 1, Definition 1.1]{Shalit} for the monoid $\IZ$,  and then we present the Motzkin Jones-Wenzl idempotents 
  from \cite{Jo21}.

 \begin{definition}\label{defsub}
A  subproduct system (over the additive monoid $\IZ_+$) is a sequence of Hilbert spaces $\{H_k\}_{k=0}^\infty$ such that $H_0=\IC$, 
$\dim H_1<+\infty$, along with a family of 
co-products, namely 
isometries $w_{k,l}: H_{k+l}\to H_k\otimes H_l$ satisfying
the following co-associativity condition
$$
(w_{k,l}\otimes 1)w_{k+l,n}= (1\otimes w_{l,n})w_{k,l+n}\qquad k, l, n\in \IZ_+\, . 
$$
Here the maps $w_{0,n}$  and $w_{n,0}$ are the canonical identifications.
\end{definition}

By \cite[Lemma 6.1]{Shalit}, we do not harm generality in assuming that $H_0=\IC$, 
$H_p\subset (H_1)^{\otimes p}$,
$H_{p+q}\subset H_p\otimes H_q$,
and the isometries $w_{k,l}$ are the embedding maps. Subproduct systems realized in this way 
are known as \textbf{standard subproduct systems}.
If we consider the projections $f_n: (H_1)^{\otimes n}\to H_n\subset (H_1)^{\otimes n}$,
the co-associativity condition can be recast as
\begin{align}
&f_n\otimes 1\geq f_{n+1}\leq 1\otimes f_n \qquad \text{for all } n\in \IZ_+, \label{coass-proj}
\end{align}
as shown in  \cite[Theorem 7.2]{GerholdSkeide}.
 
Starting with a finite-dimensional Hilbert space $H$ of dimension $n$,
we can consider a standard subproduct system  with $H_1=H$.
We may then construct the associated Fock space $\CF_\CH :=\oplus_{k\geq 0} H_k$, 
 as a subspace of the full Fock space $\CF(H):=\oplus_{k\geq 0} H^{\otimes k}$.
To this aim, we fix  an orthonormal basis $\{v_i\}_{i=1}^n$ of $H_1$. 
Denote by $e_\CH$ the orthogonal projection of $\CF(H)$ onto $\CF_\CH$.
Fix an orthonormal basis $\{v_i\}_{i=1}^n$ of $H_1$.
For any $v\in H_1$, define the creation operators $T_v: \CF(H)\to \CF(H)$ and
$S_v: \CF_\CH\to \CF_\CH$ by 
\begin{equation}
\begin{aligned}
&T_v u := v \otimes u, \qquad v \in H_1, \; u \in \CF(H)\\
&T_i:=T_{v_i}\\
&S_i:= e_\CH T_i \upharpoonright_{\CF_\CH}\\
&S_i^* = e_\CH T_i^* \upharpoonright_{\CF_\CH}=T_i^* \upharpoonright_{\CF_\CH}\,.
\end{aligned}
\end{equation}

Note that 
$$
1-\sum_{i=1}^n S_iS_i^* = e_0\,,
$$ 
where $e_0$ is the orthogonal projection of  $\CF_\CH$ onto $H_0$. 
\begin{definition}
We denote by $\CT_\CH$ the unital C$^*$-algebra generated by $\{S_v$ : $v\in H_1\}$, which is known as the
 Toeplitz algebra 
$\CT_\CH$
associated with $\CH$.

As $1\in H_0$ is cyclic for $\CT_\CH$,
 the ideal of compact operators $\CK(\CF_\CH)$ is contained in  $\CT_\CH$.
The corresponding quotient 
 $\CO_\CH:=\CT_\CH/\CK(\CF_\CH)$ is the so-called  the Cuntz-Pimsner algebra.  
 \end{definition}
 In the simplest case,  $H_k=H^{\otimes k}$, $\CO_\CH$ returns the well known Cuntz algebra \cite{Cuntz}.
 Cuntz-Pimsner algebras for subproduct systems were introduced by Viselter in \cite{Viselter12}, generalizing the corresponding concept  for C$^*$-correspondences.
We will refer to this realization of $\CT_\CH$ as the canonical representation. 

 We now move to the Motzkin Jones-Wenzl idempotents $g_k$ of the Motzkin algebra.
Let us   first recall the definition of $g_k$ and prove some of their basic properties. 
\begin{definition}\cite[Section 2.4]{Jo21}
 Let $P_m(x)$, $n\geq 0$,  be the Chebyshev polynomials over $\IC$ defined as $P_0(x)=P_1(x)=1$, $P_{m+1}(x)=P_m(x)-xP_{m-1}(x)$. 
$\lambda^{-1}\in \IC$ is said to be $n$-generic if $P_k((\lambda^{-1}-1)^{-2})\neq 0$ for $1\leq k \leq n$ and generic if $P_k((\lambda^{-1}-1)^{-2})\neq 0$ for all $k\in\IN_0$.
For $\lambda^{-1}$ $n$-generic the Motzkin Jones-Wenzl idempotents $g_k\in M_k(\lambda^{-1})$, for $1\leq k \leq n-1$, are defined inductively as 
\begin{align}\label{recurrence-formula}
g_{k+1}&=g_k(1-p_{k+1})-\frac{\lambda^{-1}}{\lambda^{-1}-1} \frac{P_{k-1}((\lambda^{-1}-1)^{-2}))}{P_k((\lambda^{-1}-1)^{-2}))} g_k t_kg_k\,,
\end{align}
where $g_1=1-p_1$.
\end{definition}

In this paper, we study the subproduct system given by $H_0=\IC$ and $H_k:=g_{k}H^{\otimes k}$ for $k\geq 1$, 
 where the Motzkin algebra $M_k(\lambda^{-1})$ is thought of as being realised concretely in the representation of Proposition \ref{extended-rep}.
 From now on we restrict our analysis to $\lambda\in (0,1/3]$, where the Motzkin algebra affords a C$^*$-algebraic structure,  see \cite{Jo21}.
 The proof of the results in the present section work also for $\lambda^{-1}$ being real $n$-generic for all $n$.
 The other section, however, exploits the specific range of parameters, see e.g. Lemma \ref{lemma31}.
   \begin{proposition}\cite[Prop. 2.1]{Jo21} \label{propJY}
Let $E: M_k(\lambda^{-1})\to M_{k-1}(\lambda^{-1})$ be the map defined in \eqref{conditionalexpectation}. 
Then the Motzkin Jones-Wenzl idempotents enjoy the
 following properties
\begin{enumerate}
\item $g_k$ is the unique orthogonal projection   such that  $g_kp_i=p_ig_k=0$ for $i=1, \ldots, k-1$;
\item  $E(g_k)=\frac{(\lambda^{-1}-1) P_{k}( (\lambda^{-1}-1)^{-2})	}{\lambda^{-1}P_{k-1}((\lambda^{-1}-1)^{-2})} g_{k-1}$;
\item  $g_ig_j=g_j$ if $1\leq i\leq j\leq n$.
\end{enumerate}
\end{proposition}

\begin{remark}
Denote by $\sigma: M_k(\lambda^{-1})\to M_k(\lambda^{-1})$ the automorphism defined on the level of diagrams as the reflection about a vertical line. 
By Proposition \ref{propJY}-(1), it follows that $\sigma(g_k)=g_k$ for all $k$. Therefore, the following recursive formula holds 
\begin{equation}\label{symmetric-recurrence-formula}
\begin{aligned}
g_{k+1}&=(\id \otimes g_k)((\id -p_{1})\otimes \id^{\otimes k})\\
&\qquad -\frac{\lambda^{-1}}{\lambda^{-1}-1} \frac{P_{k-1}((\lambda^{-1}-1)^{-2}))}{P_k((\lambda^{-1}-1)^{-2}))} (\id \otimes g_k) t_{1} (\id \otimes g_k)\,,
\end{aligned}
\end{equation}
where we used   $\sigma(g_k \otimes \id)=\id \otimes g_k$, $\sigma(t_k)=t_{1}$, $\sigma(\id^{\otimes k+1}-p_n)=(\id-p_1)\id^{\otimes k}$.
\end{remark}

\begin{lemma}\label{prop-gk}
One has  $g_{p+q}(\id^{\otimes p} \otimes g_q)=g_{p+q}(g_q\otimes \id^{\otimes p})=g_{p+q}$ if $1\leq p, q$.
\end{lemma}
\begin{proof} 
Consider the automorphism $\sigma: M_{p+q}(\lambda^{-1})\to M_{p+q}(\lambda^{-1})$
mentioned above.
Then the claim follows from Proposition \ref{propJY}-(3).
\end{proof}
We are now in a position to prove that our family of Hilbert spaces is indeed a standard product system.
\begin{proposition}
The family of Hilbert spaces 
 $H_k:=g_{k}H^{\otimes k}$
is a standard subproduct system. 
\end{proposition}
\begin{proof}
By Proposition \ref{propJY}-(3)  and Lemma \ref{prop-gk} we have
 $$
 g_{p+q}=(g_{p}\otimes \id^{\otimes q}) (\id^{\otimes p}\otimes g_q)g_{p+q}\,.
 $$
In particular, one has $g_{p+q}\leq (g_{p}\otimes \id^{\otimes q}) (\id^{\otimes p}\otimes g_q)$ and
\begin{align*}
 g_{p+q}H^{\otimes p+q}&\subset (g_{p}\otimes \id^{\otimes q}) (\id^{\otimes p}\otimes g_q) (H^{\otimes p+q})
 =  g_{p}H^{\otimes p}\otimes   g_qH^{\otimes q}\,.
\end{align*}
We note that this inclusion also implies   $H_n\subset (H_1)^{\otimes n}$.
\end{proof}
\begin{remark}
Since $H_k\subset H_1^{\otimes k}\subset H^{\otimes k}$, if we consider the product system $H_k:=g_k H_1^{\otimes k}$ we obtain the same Hilbert spaces.
From now on, we will work with them. 
\end{remark}
Set
$$
[n]_q = \left\{
\begin{array}{ll}
\frac{q^n - q^{-n}}{q - q^{-1}} & \text{if } q \in (0,1), \\
n & \text{if } q = 1.
\end{array}
\right.
$$

\begin{lemma}\label{lemma_dimension}
For $\dim H=n$,
one has
\begin{enumerate}
\item  $q_k:=\frac{\lambda^{-1}}{\lambda^{-1}-1} \frac{P_{k-1}((\lambda^{-1}-1)^{-2}))}{P_k((\lambda^{-1}-1)^{-2}))} g_k t_kg_k$ is a projection,
\item  $\frac{\lambda^{-1}}{\lambda^{-1}-1} \frac{P_{k-1}((\lambda^{-1}-1)^{-2}))}{P_k((\lambda^{-1}-1)^{-2}))} g_k t_kg_k$ is equivalent to $g_{k-1}p_kp_{k+1}$,
\item $\dim H_{k+1}=(n-1)\dim H_k - \dim H_{k-1}$, with $\dim H_1= n-1$,
\item $\dim H_k = [k+1]_t$, where $t\in (0,1]$ is such that $t+t^{-1}=\dim H_1=n-1$. 
\end{enumerate}
\end{lemma}
\begin{proof}
(1) By induction it is easy to see that all the coefficients of $P_k(x)$ are real and thus $q_k$ is self-adjoint.
We show that $q_k$ is also idempotent
\begin{align*}
q_kq_k&= \left( \frac{\lambda^{-1}}{\lambda^{-1}-1} \frac{P_{k-1}((\lambda^{-1}-1)^{-2}))}{P_k((\lambda^{-1}-1)^{-2}))} \right)^2 g_k t_kg_k    g_k t_kg_k  \\
&= \left( \frac{\lambda^{-1}}{\lambda^{-1}-1} \frac{P_{k-1}((\lambda^{-1}-1)^{-2}))}{P_k((\lambda^{-1}-1)^{-2}))} \right)^2 g_k t_kg_k      t_kg_k\,.
\end{align*}
By multiplying the factors in $g_kt_kg_kt_kg_k$, one recognises that in the middle there is a multiple of $E(g_k)$ and that was calculated in Proposition \ref{propJY}-(2), then one may replace $g_kg_{k-1}$ by $g_k$ thanks to  Proposition \ref{propJY}-(3)
\[
\begin{tikzpicture} 
\node at (0.5,0.25) {$\scalebox{1}{$g_kt_kg_kt_kg_k=\lambda^2$}$};

\draw (2.3, 1) rectangle (3.3, 1.5);
\node at (2.75,1.2) {$\scalebox{.85}{$g_k$}$};
\draw (3.5,1) -- (3.5,1.7);
\draw (2.4,1.5) -- (2.4,1.7);
\draw (3.1,1.5) -- (3.1,1.7);

\draw (3.25,1) to[out=-90,in=-90] (3.5,1);
\draw (2.4,1) -- (2.4,.5);
\draw (3.1,1) -- (3.1,.5);
\node at (2.75,.75) {$\scalebox{.85}{$\ldots$}$};
 \draw (3.25,.5) to[out=90,in=90] (3.5,.5);

\draw (2.3, 0) rectangle (3.3, 0.5);
\node at (2.75,.2) {$\scalebox{.85}{$g_k$}$};
\draw (3.5,0) -- (3.5,.5);

\draw (3.25,0) to[out=-90,in=-90] (3.5,0);
\draw (3.25,-0.5) to[out=90,in=90] (3.5,-0.5);
\draw (2.4,0) -- (2.4,-.5);
\draw (3.1,0) -- (3.1,-.5);
\node at (2.75,-.25) {$\scalebox{.85}{$\ldots$}$};

\draw (2.3, -0.5) rectangle (3.3, -1);
\node at (2.75,.-.75) {$\scalebox{.85}{$g_k$}$};
\draw (3.5,-1.2) -- (3.5,-.5);
\draw (2.4,-1) -- (2.4,-1.2);
\draw (3.1,-1) -- (3.1,-1.2);

\draw[dashed] (2.2, -0.15) rectangle (3.7, 0.65);
\node at (4,0.25) {$\scalebox{1}{$=$}$};

 \end{tikzpicture}
\begin{tikzpicture} 
\node at (0,0.25) {$\scalebox{1}{$\frac{(\lambda^{-1}-1) P_{k}( (\lambda^{-1}-1)^{-2})	\lambda}{\lambda^{-1}P_{k-1}((\lambda^{-1}-1)^{-2})} $}$};

\draw (2.3, 1) rectangle (3.3, 1.5);
\node at (2.75,1.2) {$\scalebox{.85}{$g_k$}$};
\draw (3.5,1) -- (3.5,1.7);
\draw (2.4,1.5) -- (2.4,1.7);
\draw (3.1,1.5) -- (3.1,1.7);

\draw (3.25,1) to[out=-90,in=-90] (3.5,1);
\draw (2.4,1) -- (2.4,.5);
\draw (3.1,1) -- (3.1,.5);
\node at (2.75,.75) {$\scalebox{.85}{$\ldots$}$};

\draw (2.3, 0) rectangle (3.0, 0.5);
\node at (2.65,.2) {$\scalebox{.85}{$g_{k-1}$}$};
\draw (3.1,0) -- (3.1,.5);
\draw (3.25,0) -- (3.25,.5);
\draw (3.5,0) -- (3.5,.5);

\draw (3.25,-0.5) to[out=90,in=90] (3.5,-0.5);
\draw (2.4,0) -- (2.4,-.5);
\draw (3.1,0) -- (3.1,-.5);
\node at (2.75,-.25) {$\scalebox{.85}{$\ldots$}$};

\draw (2.3, -0.5) rectangle (3.3, -1);
\node at (2.75,.-.75) {$\scalebox{.85}{$g_k$}$};
\draw (3.5,-1.2) -- (3.5,-.5);
\draw (2.4,-1) -- (2.4,-1.2);
\draw (3.1,-1) -- (3.1,-1.2);

 \end{tikzpicture}
\]
\[
\begin{tikzpicture} 
\node at (0,0.25) {$\scalebox{1}{$=\frac{(\lambda^{-1}-1) P_{k}( (\lambda^{-1}-1)^{-2})	 }{\lambda^{-1} P_{k-1}((\lambda^{-1}-1)^{-2})} \lambda$}$};

\draw (2.3, 1) rectangle (3.3, 1.5);
\node at (2.75,1.2) {$\scalebox{.85}{$g_k$}$};
\draw (3.5,1) -- (3.5,1.7);
\draw (2.4,1.5) -- (2.4,1.7);
\draw (3.1,1.5) -- (3.1,1.7);

\draw (2.4,1) -- (2.4,.5);
\draw (3.1,1) -- (3.1,.5);
\node at (2.75,.75) {$\scalebox{.85}{$\ldots$}$};
 
\draw (3.25,-0.2) -- (3.25,1);
\draw (3.5,-.2) -- (3.5,1);
\draw (3.25,-.2) to[out=-90,in=-90] (3.5,-.2);

\draw[dashed] (2.2, -0.1) rectangle (3.7, 1.6);

\draw (2.3, 0) rectangle (3.0, 0.5);
\node at (2.65,.2) {$\scalebox{.85}{$g_{k-1}$}$};
\draw (3.1,0) -- (3.1,.5);

\draw (3.25,-0.5) to[out=90,in=90] (3.5,-0.5);
\draw (2.4,0) -- (2.4,-.5);
\draw (3.1,0) -- (3.1,-.5);
\node at (2.75,-.25) {$\scalebox{.85}{$\ldots$}$};

\draw (2.3, -0.5) rectangle (3.3, -1);
\node at (2.75,.-.75) {$\scalebox{.85}{$g_k$}$};
\draw (3.5,-1.2) -- (3.5,-.5);
\draw (2.4,-1) -- (2.4,-1.2);
\draw (3.1,-1) -- (3.1,-1.2);

 \end{tikzpicture}
\]
which yield
$\frac{(\lambda^{-1}-1) P_{k}( (\lambda^{-1}-1)^{-2})	}{\lambda^{-1}P_{k-1}((\lambda^{-1}-1)^{-2})} g_kt_kg_k$.

(2)
Take $x:=\sqrt{\frac{\lambda^{-1}}{\lambda^{-1}-1} \frac{P_{k-1}((\lambda^{-1}-1)^{-2}))}{P_k((\lambda^{-1}-1)^{-2}))}} g_k y$, where 
\[
\begin{tikzpicture}

\draw (3.25,0) to[out=-90,in=-90] (3.5,0);
\draw (2.4,0) -- (2.4,-.5);
\draw (3.1,0) -- (3.1,-.5);
\node at (1.35,-0.25) {$\scalebox{.85}{$y=\sqrt{\lambda}$}$};
\node at (2.75,-.25) {$\scalebox{.85}{$\ldots$}$};

 \end{tikzpicture}
\]
Clearly, one has $xx^*=q_k$. By using Proposition \ref{propJY}-(2), one also has
\[
\begin{tikzpicture} 
\node at (-.75,0.25) {$\scalebox{1}{$x^*x=\frac{\lambda^{-1}}{\lambda^{-1}-1} \frac{P_{k-1}((\lambda^{-1}-1)^{-2}))\lambda}{P_k((\lambda^{-1}-1)^{-2}))}$}$};



 \draw (2.4,1) -- (2.4,.5);
\draw (3.1,1) -- (3.1,.5);
\node at (2.75,.75) {$\scalebox{.85}{$\ldots$}$};
 \draw (3.25,.5) to[out=90,in=90] (3.5,.5);

\draw (2.3, 0) rectangle (3.3, 0.5);
\node at (2.75,.2) {$\scalebox{.85}{$g_k$}$};
\draw (3.5,0) -- (3.5,.5);

\draw (3.25,0) to[out=-90,in=-90] (3.5,0);
 \draw (2.4,0) -- (2.4,-.5);
\draw (3.1,0) -- (3.1,-.5);
\node at (2.75,-.25) {$\scalebox{.85}{$\ldots$}$};
 
\draw[dashed] (2.2, -0.15) rectangle (3.7, 0.65);
\node at (4,0.25) {$\scalebox{1}{$=$}$};

 \end{tikzpicture}
\]
\begin{align*}
&=\frac{\lambda^{-1}}{\lambda^{-1}-1} \frac{P_{k-1}((\lambda^{-1}-1)^{-2}))\lambda}{P_k((\lambda^{-1}-1)^{-2}))}\cdot \frac{(\lambda^{-1}-1) P_{k}( (\lambda^{-1}-1)^{-2})	}{\lambda^{-1}P_{k-1}((\lambda^{-1}-1)^{-2})} g_{k-1}  p_{k+1} p_k \\
& =g_{k-1} p_{k+1} p_k\,,
\end{align*}
where in the last step we used that the elements belong to $M_{k+1}(\lambda^{-1})$.

(3) Note that $g_{k-1}p_kp_{k+1} H^{\otimes k+1}= g_{k-1} H^{\otimes k-1} \otimes \IC v \otimes \IC v$ and
$$
g_k H^{\otimes k+1} = (g_k H^{\otimes k})\otimes H \qquad\qquad  g_k p_{k+1} H^{\otimes k+1} = (g_k H^{\otimes k})\otimes \IC v\,.
$$
By (1), (2), and \eqref{recurrence-formula} one gets
$$
\dim H_{k+1}=(n-1)\dim H_k - \dim H_{k-1}\,.
$$
For the dimension of $H_1=g_1 H= (1-p)H$, one has 
$$
\dim H_1 = \dim H - \dim pH=\dim H -1\,,
$$
and in particular $\dim g_1 H = n-1$.

(4) Our claim follows from \cite[Lemma 1.6]{Nesh}, where the product system satisfies the equation
$\dim H_{k+1} =\dim H_1 \dim H_k- \dim H_{k-1}$, with $\dim H_1=n-1$.
There it was proved that the solution for that equation is $[k+1]_t$ where $t\in (0,1]$ is such that $t+t^{-1}=\dim H_1$.
\end{proof}

In general, it is   possible to give an equivalent characterization of standard subproduct system with $H_1=H$ that we now recall. 
Fix a basis $\{v_i\}_{i=1}^n$ of $H$ and identify the tensor algebra $T(H)$
with the algebra of non-commutative polynomials $\IC\langle X_1, \ldots, X_n\rangle$.
By \cite[Prop. 7.2]{Shalit} there is a bijection between standard subproduct systems with  
$H_1=H$ and homogeneous ideals of $\IC\langle X_1, \ldots, X_n\rangle$ provided by the map $I\to \{H_k\}_{k\geq 0}$, where $H_k:=I_k^\bot\subset H^{\otimes k}$.

Given a Motzkin pair $(v_A, v)$, we recall from Definition \ref{defsupport} that $\supp(v)$ is the set of $i$'s such that $\langle v_i, v\rangle \neq 0$.
We denote by $[n]:=\{1, \ldots, n\}$.
\begin{proposition}\label{description_ideals}
Let $H$ be an $n$-dimensional Hilbert space, 
$(v_A,v)$ be a Motzkin pair and denote by $J:=\supp(v)$. $J$ has at least one element $j_0$. 
Define the subspace
 $H_J:=\{u\in H \, : \, \langle v_i, u\rangle = 0 \, \, \, \forall i \in J\}$
and choose an orthonormal basis $\{w_j\}_{j\in J\setminus \{j_0\}}$ of $v^\bot \cap H_J^\bot$ 
with $\supp(w_j)\subseteq J$ for all $j\in J\setminus \{j_0\}$.
For   $i\in J\setminus \{j_0\}$, write
  $(\id-p)v_i=\sum_{j\in J\setminus \{j_0\}} c_{ji} w_j$. 
Then our given subproduct system is the one arising from the ideal 
$$
\langle \sum_{i\in [n]\setminus J} a_i X_i X_{\bar i}+ \sum_{i\in J, j, j'\in J\setminus \{j_0\}} a_{i}c_{j,i} c_{j', \bar{i}} Y_{j}Y_{j'} \rangle
$$
where $X_i=v_i$ and $Y_j=w_j$ for $i\in [n]\setminus J$, $j\in J\setminus \{j_0\}$.
\end{proposition}
\begin{proof}
By Proposition \ref{propJY}-(1) the subspace $g_2H^{\otimes 2}$ is orthogonal to $v_A$.
As $(\id-p)^{\otimes 2}g_2=g_2$, also $(\id-p)^{\otimes 2}v_A$ is orthogonal to $g_2H^{\otimes 2}$.
If we set $w_{j_0}:=v$, then we have $v_i=\sum_{j\in J} c_{j,i} w_j$ for all $i\in J$ and
some unitary matrix $(c_{ji})$.
We thus have 
$$
(\id-p)v_i=\sum_{j\in J\setminus \{j_0\}} c_{j,i} w_j\,,
$$
and
\begin{align*}
 (\id-p)^{\otimes 2}v_A&=  
  \sum_{i\in [n]\setminus J} a_i v_i \otimes v_{\bar i}+ \sum_{i\in J, j, j'\in J\setminus \{j_0\}} a_{i}c_{j,i} c_{j', \bar{i}} w_{j} \otimes w_{j'}   
\end{align*}
which is contained in $(g_2H^{\otimes 2})^\bot$.
By Lemma \ref{lemma_dimension}, we have $\dim H_2=(n-1)^2-1$, 
which means
that
$$
H_2 = {\rm span} \left\{ \sum_{i\in [n]\setminus J} a_i v_i \otimes v_{\bar i}+ \sum_{i\in J, j, j'\in J\setminus \{j_0\}} a_{i}c_{j,i} c_{j', \bar{i}} w_{j} \otimes w_{j'} \right\}^\bot
$$
That claim holds for  $H_n$
follows by the identity
$$
g_n(H^{\otimes n}) = H_1^{\otimes n}\ominus \left\{ (1 - p)^{\otimes n}  {\rm span} \left\{ \Im(t_i), \Im(p_j) \, | \, 1\leq i\leq n-1, 1,\leq j\leq n\right\}\right\}
$$
which can be established by a direct computation.
\end{proof} 

We now move on to introduce the wider class of subproduct systems  we alluded to in the introduction. 
 We continue the discussion begun in Example \ref{seriesofexamples}, retaining the same numbering.
\begin{example}\label{seriesofexamples-part2}
Let $H$ be an $n$ dimensional Hilbert space, 
with $n\geq 3$,
$v_A$ be a unit Temperley-Lieb vector (here $A\in M_{n}(\IC)$). 
 \begin{enumerate}
\item[i)]  Here we assume that $\dim H=n=2m+1$ and choose $v=v_{m+1}$. Consider the Motzkin algebra   represented as in Example \ref{seriesofexamples}-(i). As an orthonormal basis for $H_1$ we take $\{v_i \, | \, 1\leq i \leq n, i\neq m+1\}$.
Then the ideal of $\IC\langle X_i\; | \; 1\leq i \leq n, i\neq m+1\rangle$ corresponding to the
 standard subproduct system   $H_r:=g_r H_1^{\otimes r}$, is 
$$
\langle \sum_{i\in [n]\setminus \{m+1\}} a_i X_i X_{\bar i}  \rangle
$$ 
\item[ii)] 
If $A$ is such that  $a_1=a_n$, $|a_1|=\sqrt{\lambda}$, and $b_i= (\sqrt{2})^{-1}\delta_{1,i}+ (\sqrt{2})^{-1}\delta_{n,i}$, as in Example \ref{seriesofexamples}-(ii). For a suitable orthonormal basis $w_1\cup \{v_2\}_{i=2}^{n-1}$ of $H_1$, 
one may identify $X_1=w_1=(\sqrt{2})^{-1}v_1- (\sqrt{2})^{-1} v_n$, $X_2=v_2$, \ldots $X_{n-1}=v_{n-1}$.
 Then
the ideal of $\IC\langle X_i\; | \; 1\leq i \leq n-1 \rangle$ corresponding to the
 standard subproduct system   $H_k:=g_k H_1^{\otimes k}$ is
$$
\langle -  a_1    X_1 X_1+ \sum_{i=2}^{n-1} a_i X_iX_{\bar i}\rangle
$$

 \item[iii)] For for any $r\leq n/2$,  suppose  $A$
 satisfies
  $a_1=\ldots =a_r=a_{\bar{1}}=\ldots =a_{\bar{r}}$ and $|a_1|=\lambda^{1/2}$, and
  $b_i:=(\sum_{j=1}^r \delta_{i,j}+ \delta_{i,\bar{j}})/\sqrt{2r}$ (as in Example \ref{seriesofexamples}-(iii)). 
For a suitable orthonormal basis $\{w_i\}_{i=1}^{2r-1}\cup \{v_i\}_{i=r+1}^{n-r}$ of $H_1$, 
one may identify $Y_1=w_1$, \ldots, $Y_{2r-1}=w_{2r-1}$,
 $X_1=v_{r+1}$, \ldots, $X_{n-2r}=v_{n-r}$. 
 Then
the ideal of $\IC\langle X_i, Y_j\; | \; 1\leq i \leq n-2r,  1\leq j \leq 2r \rangle$ corresponding to the
 standard subproduct system   $H_k:=g_k H_1^{\otimes k}$ is
\begin{align*}
H_2 = \rm{span} &\left\{ a_1   \sum_{b=1}^{2r-1}   \langle 1,-b\rangle  Y_b
    Y_{b}   + \sum_{j=r+1}^{n-r} a_j X_j  X_{\bar j}
 \right\}^\bot
\end{align*}
where $\langle a,b\rangle := e^{2\pi iab/(2r)}$ for $a,b\in\{1, \ldots, 2r\}$.

This can be shown by the following computations.\\
Consider the matrix $W=((\sqrt{2r})^{-1} \langle a,b\rangle)\in M_{2r,2r}(\IC)$, whose inverse is $T=((\sqrt{2r})^{-1} \langle a,-b\rangle)$.
Define the vectors $w_1, \ldots, w_{2r}\in H$ by
$$
(v_1, \ldots, v_r, v_{\bar{r}}, \ldots, v_{\bar 1})W=(w_1, \ldots, w_{2r})
$$
i.e. $w_a$ is the vector whose $b$-th entry is 
the pairing $\langle a,b\rangle$ up to normalization.
Note $w_{2r}=v$ and that $\{w_i\}_{i=1}^{2r}$ is an orthonormal family in $H$.
Clearly, we have
$$
(v_1, \ldots, v_r, v_{\bar{r}}, \ldots, v_{\bar 1})=(w_1, \ldots, w_{2r})T
$$
We rename $(v_1, \ldots, v_r, v_{\bar{r}}, \ldots, v_{\bar 1})$ as $(v_{j_1}, \ldots, v_{j_{2r}})$
(in the same order).
In particular, for $ a\in \{1, \ldots, 2r\}$, we have 
\begin{align*}
(\id-p)v_{j_a}&=\sum_{b=1}^{2r-1} (\sqrt{2r})^{-1} \langle a,-b\rangle w_b \\
(\id-p)v_{\bar j_a}&=(\id-p)v_{n-a+1}=\sum_{b=1}^{2r-1} (\sqrt{2r})^{-1} \langle 2r-a+1,-b\rangle w_b\,. 
\end{align*}
By taking the same steps as in the proof of Proposition \ref{description_ideals}, we have that
\begin{align*}
&(\id -p)^{\otimes 2}(v_A)=\\
&=\sum_{q=1}^{2r} a_1 (\id -p)v_{j_q} \otimes (\id -p)v_{\bar j_q} + \sum_{j=r+1}^{n-r} a_j v_j\otimes  v_{\bar{j}}\\
&=\sum_{q=1}^{2r} \frac{a_1}{2r}  \left(  \sum_{b=1}^{2r-1} \langle q,-b\rangle w_b
 \right) \otimes 
 \left( \sum_{b'=1}^{2r-1} ( \langle 2r-q+1,-b'\rangle w_{b'}
  \right) + \sum_{j=r+1}^{n-r} a_j v_j\otimes  v_{\bar{j}}\\
&=\sum_{q=1}^{2r} \frac{a_1}{2r}    \sum_{b, b'=1}^{2r-1} \langle q,-b\rangle   \langle 2r-q+1,-b'\rangle w_b
  \otimes  w_{b'}
   + \sum_{j=r+1}^{n-r} a_j v_j\otimes  v_{\bar{j}}\\
&= \frac{a_1}{2r}    \sum_{b, b'=1}^{2r-1} \left( \sum_{q=1}^{2r} \langle q,-b\rangle   \langle -q+1,-b'\rangle\right) w_b
  \otimes  w_{b'}
   + \sum_{j=r+1}^{n-r} a_j v_j\otimes  v_{\bar{j}}\\
&= \frac{a_1}{2r}    \sum_{b, b'=1}^{2r-1} \left( \sum_{q=1}^{2r} \langle q,-b\rangle  \langle q,b'\rangle  \right) \langle 1,-b'\rangle w_b
  \otimes  w_{b'}   + \sum_{j=r+1}^{n-r} a_j v_j\otimes  v_{\bar{j}}\\
&= a_1   \sum_{b=1}^{2r-1}   \langle 1,-b\rangle  w_b
  \otimes  w_{b}   + \sum_{j=r+1}^{n-r} a_j v_j\otimes  v_{\bar{j}}\,,
\end{align*}
so
\begin{align*}
H_2 = \rm{span} &\left\{ 
a_1   \sum_{b=1}^{2r-1}   \langle 1,-b\rangle  w_b
  \otimes  w_{b}   + \sum_{j=r+1}^{n-r} a_j v_j\otimes  v_{\bar j}  \right\}^\bot .
\end{align*}
\end{enumerate}
\end{example}
\begin{remark}
Up to a change of basis and possibly changing one entry of the matrix defining the Temperley-Lieb vector, we may recover the subproduct systems 
of \cite{Nesh}.
Indeed, in Example \ref{seriesofexamples-part2}-(i)  gives back the Temperley-Lieb subproduct of \cite{Nesh} produced by an 
even dimensional Hilbert space,
when  $\dim H$ is odd, i.e. $\dim H_1$ even.
Likewise Example \ref{seriesofexamples-part2}-(ii) returns the Temperley-Lieb subproduct system of \cite{Nesh}
produced by an odd dimensional Hilbert space,
when $\dim H$ is even, namely $\dim H_1$ is odd. 
However, if  $\dim H$ is odd, the subproduct system of Example \ref{seriesofexamples-part2}-(ii) is novel.
\end{remark}

\section{The Toeplitz and Cuntz-Pimsner C$^*$-algebras}\label{sec3}
The goal of this section is to   show that the Toeplitz 
and the Cuntz-Pimsner 
algebras associated with Example  \ref{seriesofexamples-part2}-(iii) 
can be seen
 as universal C$^*$-algebras, 
 defined in terms of a finite set of generators and relations (see Definitions \ref{defunitoep}, \ref{universalOP}).
 Furthermore, we investigate their representation theory.

We begin by showing some relations that the generators of our algebra satisfy.

We denote the orthogonal projection of $\CF_\CH$ onto  $H_n$ by $e_n$.
In the sequel, we identify the unital subalgebra
$C^*(\{ e_n  \, | \, n\in \IZ_+\})$
with $C(\IZ_+\cup \{\infty\})$ through the isomorphism 
 $e_n\mapsto \mathds{1}_n$ for all $n\in \IZ_+$. 
Denote by $\gamma$ the shift to the left in $C(\IZ_+\cup \{\infty\})$, that is the map $\gamma(f)(n)=f(n+1)$, $f\in C(\IZ_+\cup \{\infty\})$, in other words $\gamma(e_k)=e_{k-1}$ for any $k\geq 1$.

Define $\varphi: \IZ_+\cup\{\infty\}\to\IC$ as 
$$
\varphi(m):=\frac{\lambda^{-1}}{\lambda^{-1}-1} \frac{P_{m-1}((\lambda^{-1}-1)^{-2})}{P_m((\lambda^{-1}-1)^{-2}))}\,.
$$
\begin{lemma}\label{lemma31}
For 
 $q\in (0,1]$ such that $q+q^{-1}=\lambda^{-1}-1$,
one has
$$
\lim_{m\to \infty} \varphi(m) 
= q^2+q+1\,.
$$
In particular, $\varphi\in C(\IZ_+\cup \{\infty\})$.
\end{lemma}
\begin{proof}
Let $Q_m(y)\in \IC[y]$ be the polynomial defined by the equation $Q_{m+1}(y)=yQ_m(y)-Q_{m-1}(y)$ and $Q_0(y)=1$, $Q_1(y)=y$.
If $y=q+q^{-1}$, then an explicit form for $Q_m$ is 
$Q_m(y)=[m+1]_q$, \cite{Nesh}.
It is easy to check that for $0<|q|\leq 1$,
 $$\lim_m \frac{Q_{m-1}(y)}{Q_{m}(y)}=q.$$

The function $P_m(x)$ is obtained by taking $x^{m/2}Q_m(x^{-1/2})$.
For $x=(\lambda^{-1}-1)^{-2}$, $q+q^{-1}=\lambda^{-1}-1$
$$
\lim_{m\to \infty}\frac{P_{m-1}((\lambda^{-1}-1)^{-2})}{P_m((\lambda^{-1}-1)^{-2}))}= \lim_{m\to \infty}  (q+q^{-1}) \frac{Q_{m-1}(\lambda^{-1}-1)}{Q_m(\lambda^{-1}-1)}= q^2+1
$$
Then
$$
\lim_{m\to \infty} \frac{\lambda^{-1}}{\lambda^{-1}-1} \frac{P_{m-1}((\lambda^{-1}-1)^{-2})}{P_m((\lambda^{-1}-1)^{-2}))}= q^2+q+1\,.
$$
\end{proof}
For future reference, we record that 
\begin{align}\label{formulaphi}
&\varphi(m)=\lambda^{-1}\frac{[m]_q}{[m+1]_q}=\lambda^{-1}\frac{q^m-q^{-m}}{q^{m+1}-q^{-m-1}}\,.
\end{align}
Clearly, 
Examples \ref{seriesofexamples}-(i) and \ref{seriesofexamples-part2}-(i), 
Examples \ref{seriesofexamples}-(ii) and \ref{seriesofexamples-part2}-(ii) with $\dim H$ even,  
yield exactly the same C$^*$-algebras 
studied in \cite{Nesh}. As we will see in the next result, the others satisfy   different relations.
\begin{proposition}\label{prop-rel}
Let $q$ be such that $q+q^{-1}=\lambda^{-1}-1$.
Set
$J:=\{1,\ldots, r\}\cup \{\bar r,\ldots, \bar 2\}=\{j_1, \ldots, j_{2r-1}\}$, $j_{2r}:=n$,
and
$J^\circ:=\{r+1, \ldots, n-r\}$.
Under the hypotheses of Examples \ref{seriesofexamples}-(iii) and \ref{seriesofexamples-part2}-(iii),  the following relations hold
\begin{align}
&f S_i=S_i \gamma(f) \text{ for $f\in C(\IZ_+\cup \{\infty\})$}   \label{eq1rel}\\
&\sum_{i\in J^\circ \cup J} S_i S_{i}^*  = 1-e_0  \label{eq2rel}\\ 
&\sum_{b=1}^{2r-1} a_1 \langle 1,-b\rangle   S_{j_b}S_{j_{b}}    + \sum_{j=r+1}^{n-r} a_j S_j S_{\bar{j}}   =0 \label{eq3rel}\\
&   S^*_jS_i= \delta_{ij} -\varphi \bar{a}_{ i} a_jS_{\bar j}S_{\bar i}^*\qquad i, j \in J^\circ=\{r+1, \ldots, n-r\} \label{eq4rel}\\
&   S^*_{j} S_{j_s}= -   \varphi\bar{a}_{j_s} a_{j}    \langle  s, 1\rangle  S_{\bar j}S_{j_s}^* \qquad j_s\in J, j \in J^\circ
 \label{eq5rel}\\
&  S^*_{j_{s'}} S_{j_s}= \delta_{s, s'}  - \varphi \lambda 	  \langle s-s', 1\rangle  	S_{j_s}S_{j_{s'}}^* \qquad j_{s}, j_{s'} \in J=\{j_1, \ldots, j_{2r-1}\}
			\label{eq6rel}
\end{align} 
where	
$$
S_i=\left\{ 
\begin{array}{ll}
S_{w_s} & \text{if } i=j_{s}\in J\\
S_{v_i} & \text{if } i\in J^\circ\,. 
\end{array}
\right.
$$
\end{proposition}
 \begin{proof}
 Recall that $W=((\sqrt{2r})^{-1} \langle a,b\rangle)\in M_{2r}(\IC)$,
 $$
 (v_1, \ldots, v_r, v_{\bar{r}}, \ldots, v_{\bar 1})=(v_{j_1}, \ldots, v_{j_{2r}})
 $$
and that a basis of the hyperplane orthogonal to $v$, that is $H_1$, is given by
  $$
 w_h=\sum_{k=1}^{2r}\frac{\langle k,h\rangle v_{j_k}}{\sqrt{2r}}=\sum_{k=1}^{2r} c_{kh}v_{j_k}\qquad h\in\{1, \ldots , 2r-1\}\,.
 $$
 Note that 
\begin{align*}
 v_{j_h}& =\sum_{k=1}^{2r}\frac{\langle k,-h\rangle w_k}{\sqrt{2r}}=\sum_{k=1}^{2r} \bar c_{kh} w_k\qquad h\in\{1, \ldots , 2r\}\\
 (1-p)v_{j_h}& =\sum_{k=1}^{2r-1}\frac{\langle k,-h\rangle w_k}{\sqrt{2r}}=\sum_{k=1}^{2r-1} \bar c_{kh}w_k\qquad h\in\{1, \ldots , 2r\}
\end{align*}
 
 Equation \eqref{eq1rel} reflects the fact that $S_iH_k\subset H_{k+1}$. 
  Equation \eqref{eq2rel} holds for any standard subproduct system.  
As for Equation \eqref{eq3rel}, it follows from the ideal description provided   in 
 Examples  \ref{seriesofexamples-part2}-(iii). 
 
 For Equation \eqref{eq4rel}, 
 recall the recursion formula \eqref{recurrence-formula} for the Motzkin Jones-Wenzl idempotents of the Motzkin algebra
 \begin{align*}
g_{m+1}&=(\id \otimes g_m)((\id -p_{1})\otimes \id^{\otimes m}) -\varphi(n)(\id \otimes g_m) t_{1} (\id \otimes g_m)\,,
\end{align*}
where $g_1=1-p_1$ and $\varphi(m):=\lambda^{-1}(\lambda^{-1} -1)^{-1}P_{m-1}((\lambda^{-1}-1)^{-2})) (P_m((\lambda^{-1}-1)^{-2}))$.
  
By using the explicit formula of $t$ given in Lemma \ref{formula_E},
the expression of  $g_{m+1}$ reads as
   \begin{align*}
g_{m+1}&=(\id \otimes g_m)((\id -p_{1})\otimes \id^{\otimes m}) -\varphi(n)(\id \otimes g_m) t_{1} (\id \otimes g_m)\\
&=(\id \otimes g_m)((\id -p_{1})\otimes \id^{\otimes m}) \\
&\qquad -\varphi(n) (\id\otimes g_m) ( \left(		 \sum_{h,k=1}^n  \bar{a}_{  h} a_k (e_{kh}\otimes e_{\bar k \bar h})	\right) \otimes \id^{\otimes m-1})(\id\otimes g_m)\\
&=(\id \otimes g_m)((\id -p_{1})\otimes \id^{\otimes m}) \\
&\qquad - \sum_{h,k=1}^n  \varphi(n)\bar{a}_{  h} a_k e_{kh} \otimes (g_m  (e_{\bar k, \bar h}\otimes \id^{\otimes m-1})  g_m  )\,.
\end{align*} 
  
  For any $u\in H_m$, $S_i u$ has two different formulas depending on whether $i=j_s\in J$ or $i\in J^\circ$.
  In the first case, we have
   \begin{align*}
S_i u&=g_{m+1}T_i u=w_{s} \otimes u  -\sum_{h,k=1}^m  \varphi(m)\bar{a}_{  h} a_k e_{kh}w_{s} \otimes (g_m  (e_{\bar k, \bar h}\otimes \id^{\otimes m-1})  g_m  )u\\
&=w_{s} \otimes u  -\sum_{h,k=1}^m  \varphi(m)\bar{a}_{  h} a_k \left( \sum_{k'=1}^{2r} c_{k's} e_{kh}v_{j_{k'}} \right) \otimes (g_m  (e_{\bar k, \bar h}\otimes \id^{\otimes m-1})  g_m  )u\\
&=w_{s} \otimes u  -\sum_{k=1}^m \sum_{k'=1}^{2r} \varphi(m)\bar{a}_{j_{k'}} a_k   c_{k's} v_{k}  \otimes (g_m  (e_{\bar k, \bar j_{k'}}\otimes \id^{\otimes m-1})  g_m  )u.
   \end{align*}
      In the second case, we have
      \begin{align*}
S_i u&=g_{m+1}T_i u=v_i \otimes u  -\sum_{h,k=1}^m  \varphi(m)\bar{a}_{  h} a_k e_{kh}v_i \otimes (g_m  (e_{\bar k, \bar h}\otimes \id^{\otimes m-1})  g_m  )u\\
&=v_i \otimes u  -\sum_{k=1}^m  \varphi(m)\bar{a}_{  i} a_k v_k \otimes (g_m  (e_{\bar k, \bar i}\otimes \id^{\otimes m-1})  g_m  )u\,.
\end{align*}

We want to compute $S^*_jS_i$. There are some cases to deal with.

\noindent
  \textbf{Case} $i,j \in J^\circ$: 
Note that $g_m(e_{ij}\otimes \id^{\otimes m -1})g_m=g_mT_iT_j^*g_m=S_iS_j^*$.\\
  We see that multiplying by $S^*_j$,    we have
  $$
  S^*_jS_ie_m= \delta_{ij}e_m  -\varphi(m)\bar{a}_{  i} a_jS_{\bar j}S_{\bar i}^*e_m\,.
  $$

\noindent
  \textbf{Case} $i=j_{s}, j=j_{s'}\in J$: 
We have
   \begin{align*}
S_j^*S_iu&=\delta_{i,j} u  -\sum_{k=1}^m \sum_{k'=1}^{2r} \varphi(m)\bar{a}_{j_{k'}} a_k   c_{k's} S_j^* v_{k}  \otimes (g_m  (e_{\bar k, \bar j_{k'}}\otimes \id^{\otimes m-1})  g_m  )u\\
&=\delta_{i,j} u  -\sum_{k=1}^{2r} \sum_{k'=1}^{2r} \varphi(m)\bar{a}_{j_{k'}} a_{j_k}   c_{k's} S_j^* v_{j_k}  \otimes (g_m  (e_{\bar j_k, \bar j_{k'}}\otimes \id^{\otimes m-1})  g_m  )u\\
&=\delta_{i,j} u  - \sum_{k, k'=1}^{2r} \varphi(m)\bar{a}_{j_{k'}} a_{j_k}   c_{k's} S_j^* \left( \sum_{k''=1}^{2r} \bar c_{kk''} w_{k''} \right)  \otimes (g_m  (e_{\bar j_k, \bar j_{k'}}\otimes \id^{\otimes m-1})  g_m  )u\\
&=\delta_{i,j} u  -  \sum_{k, k'=1}^{2r} \varphi(m)\bar{a}_{j_{k'}} a_{j_k}   c_{k's} S_j^* \left( \sum_{k''=1}^{2r-1} \bar c_{kk''} w_{k''} \right)  \otimes (g_m  (e_{\bar j_k, \bar j_{k'}}\otimes \id^{\otimes m-1})  g_m  )u\\
&=\delta_{i,j} u  -  \sum_{k, k'=1}^{2r} \varphi(m)\bar{a}_{j_{k'}} a_{j_k}   c_{k's} \bar c_{ks'}      (g_m  (e_{\bar j_k, \bar j_{k'}}\otimes \id^{\otimes m-1})  g_m  )u\\
&=\delta_{i,j} u  -  \sum_{k, k'=1}^{2r}\sum_{h, h'=1}^{2r-1}  \varphi(m)\bar{a}_{j_{k'}} a_{j_k}   c_{k's} \bar c_{ks'}      c_{hk} \bar c_{k' h'}   \langle h, -1\rangle \langle h', 1\rangle S_{j_{h'}}S_{j_h}^*u\\
&=\delta_{i,j} u  -  \sum_{k, k'=1}^{2r}\sum_{h, h'=1}^{2r-1} \varphi(m)\lambda  c_{hk} \bar c_{k' h'}  c_{k's} \bar c_{ks'}      \langle h, -1\rangle \langle h', 1\rangle S_{j_{h'}}S_{j_h}^*u\\
&=\delta_{i,j} u  -  \sum_{h, h'=1}^{2r-1}\varphi(m)\lambda  \left(  \sum_{k=1}^{2r}  c_{hk}  \bar c_{ks'}   \right) \left( \sum_{k'=1}^{2r}    \bar c_{k' h'}  c_{k's}  \right)      \langle h, -1\rangle \langle h', 1\rangle  S_{j_{h'}}S_{j_h}^* u \\
&=\delta_{i,j} u  -  \varphi(m)\lambda     \langle s-s', 1\rangle  S_{j_{s}}S_{j_{s'}}^* u\,,
  \end{align*}
where we used that,  for $j_k, j_{k'}\in J$, 
$$
g_m(e_{\bar j_k \bar j_{k'}}\otimes \id^{\otimes m-1})g_m=\sum_{h, h'=1}^{2r-1} c_{hk} \bar c_{k' h'}   \langle h, -1\rangle \langle h', 1\rangle S_{j_{h'}}S_{j_h}^*e_m
$$
 which holds because 
 $$
e_{\bar j_k \bar j_{k'}}(x)=\langle x, v_{\bar j_{k'}}\rangle v_{\bar j_k}=\sum_{h, h'=1}^{2r} c_{hk} \bar c_{k' h'}\langle x, w_h\rangle \langle h, -1\rangle \langle h', 1\rangle  w_{h'}\,.
$$ 
  
\noindent
  \textbf{Case}  $i=j_s\in J$, $j \in J^\circ$:
  \begin{align*}
S_j^*S_i u    &= -\sum_{k=1}^m \sum_{k'=1}^{2r} \varphi(m)\bar{a}_{j_{k'}} a_k   c_{k's} S_j^*v_{k}  \otimes (g_m  (e_{\bar k, \bar j_{k'}}\otimes \id^{\otimes m-1})  g_m  )u\\
&= - \sum_{k'=1}^{2r} \varphi(m)\bar{a}_{j_{k'}} a_j   c_{k's}    (g_m  (e_{\bar j, \bar j_{k'}}\otimes \id^{\otimes m-1})  g_m  )u\\
&= - \sum_{k'=1}^{2r} \sum_{k''=1}^{2r-1} \varphi(m)\bar{a}_{j_{k'}} a_j   c_{k's}      c_{2r-k'+1, k''}  S_{\bar j}S_{j_{k''}}^* u\\
&= - \sum_{k''=1}^{2r-1} \sum_{k'=1}^{2r}  \varphi(m)\bar{a}_{j_{k'}} a_j   c_{k's}      c_{2r-k'+1, k''}  S_{\bar j}S_{j_{k''}}^* u\\
&= - \sum_{k''=1}^{2r-1} \sum_{k'=1}^{2r}  \varphi(m)\bar{a}_{j_{k'}} a_j   \langle k', s \rangle      \langle  2r-k'+1, k''\rangle  (2r)^{-1}S_{\bar j}S_{j_{k''}}^* u\\
&= - \sum_{k''=1}^{2r-1} \varphi(m)\bar{a}_{j_{1}} a_j  \left( \sum_{k'=1}^{2r}   \langle k', s \rangle      \langle  k',- k''\rangle \right)   \langle  1, k''\rangle  (2r)^{-1}S_{\bar j}S_{j_{k''}}^* u\\
&= -   \varphi(m)\bar{a}_{j_{1}} a_j    \langle  1, s\rangle  S_{\bar j}S_{j_{s}}^* u\\
&= -   \varphi(m)\bar{a}_{j_{s}} a_j    \langle  1, s\rangle  S_{\bar j}S_{j_{s}}^* u\,,
  \end{align*}
where we used that 
$$
g_m  (e_{\bar j, \bar j_{k'}}\otimes \id^{\otimes m-1})  g_m   =\sum_{k''=1}^{2r-1}   c_{2r-k'+1, k''}  S_{\bar j}S_{j_{k''}}^*e_m\,,
$$
which holds because, for $j_{k'}\in J$, $j \in J^\circ$, we have
$$
e_{\bar j, \bar j_{k'}}(x)=\langle x, v_{\bar j_{k'}}\rangle v_{\bar j}=\langle x, v_{j_{2r-k'+1}}\rangle v_{\bar j}=\sum_{k''=1}^{2r} c_{2r-k'+1, k''} \langle x, w_{k''}\rangle v_{\bar j}\,.
$$
  \end{proof}

\begin{example}
For the subproduct system treated in Example \ref{seriesofexamples}-(ii) and \ref{seriesofexamples-part2}-(ii), the relations reduce to
\begin{equation}
\begin{aligned}
&f S_i=S_i \gamma(f) \text{ for $f\in C(\IZ_+\cup \{\infty\})$} \\
&\sum_{i=1}^{n-1}  S_i S_{i}^*  = 1-e_0 \\
&- a_{1} S_{1}S_{1}+\sum_{i=2}^{n-1}  a_i S_i S_{\bar i}=0 \\ 
&   S^*_jS_i= \delta_{ij} -\varphi\bar{a}_{ i} a_jS_{\bar j}S_{\bar i}^* \qquad i, j \in \{2, \ldots, n-1\} \\ 
&  S^*_jS_1= \delta_{j,1}  + \varphi \bar{a}_{  1}  a_j 		S_{\bar j}S_{1}^*\qquad j \in \{2, \ldots, n-1\} 	\\	
&  S^*_1S_1= 1 - \varphi \lambda 		S_{1}S_{1}^* \,.		
\end{aligned}
\end{equation}
\end{example}

We present the universal C$^*$-algebra that will be shown to coincide with our Toeplitz algebra.
\begin{definition}\label{defunitoep}
Let   $\tilde \CT_P$ be the universal unital C$^*$-algebra generated by
  $\tilde S_1$, \ldots, $\tilde S_{n-1}$ and $C(\IZ_+\cup \{\infty\})$ satisfying the relations \eqref{eq1rel}, \eqref{eq2rel}, \eqref{eq3rel}, \eqref{eq4rel}, \eqref{eq5rel}, \eqref{eq6rel}, with $e_0$ being understood as $\mathds{1}_0$ in $C(\IZ_+\cup \{\infty\})$. \\
 For any $\beta\in\IR$, 
 by universality the map $\sigma_\beta(\tilde S_j):=e^{2\pi i \beta}\tilde S_j$ for all $j$, $\sigma_\beta(f)=f$ 
 for all $f\in C(\IZ_+\cup \{\infty\})$,
 extends to an automorphism of $\tilde \CT_P$.
This yields an action of $\IT$ continuous w.r.t. the 
 point-wise norm convergence, which we call the gauge action. 
  \end{definition}

We now show that the canonical representation is irreducible and collect a few results on the representations of $\CT_P$ and $\CO_P$.
  \begin{lemma}\label{irreducibility_canonical_rep}
 The canonical representation of 
 $\CT_P$  is   irreducible.
 \end{lemma}
 \begin{proof}
  It suffices to show that any non-zero vector $v$ in $\CF_\CH$ is cyclic.
  Since $v\neq 0$, then $e_kv\neq 0$ for some $k\geq 0$. If $k=0$, then $e_0v\in H_0$ is cyclic.
  Suppose that $k\geq 1$. As $\sum_{i\in J^\circ \cup J} S_iS_i^*=1-e_0$, where $J^\circ=\{r+1, \ldots, n-r\}$ and $ J=\{j_1, \ldots, j_{2r-1}\}$, then $S_jS_j^*e_kv\neq 0$ for some $j$. 
  In particular, $0\neq S_j^*e_kv\in H_{k-1}$. By repeating this argument we find an $\alpha$ such that $0 \neq S_\alpha^* e_kv\in H_0$ which is  cyclic.
 \end{proof}
 \begin{remark}\label{fockSiadjoint}
 By the argument in the proof of Lemma \ref{irreducibility_canonical_rep} one obtains that  
 $$
 \sum_{i\in J^\circ \cup J}
  S_i^* \CF_P=\CF_P\,, $$
  where $J^\circ=\{r+1, \ldots, n-r\}$ and $ J=\{j_1, \ldots, j_{2r-1}\}$.
 \end{remark}

 Denote by $\tilde \CT_P^\IT\subset \tilde \CT_P$ the fixed-point subalgebra under the gauge action.
As is commonly done, we adopt the multi-index notation for monomials in both $S_i$ and $\tilde S_i$, that is for example
 $\tilde S_\alpha:=\tilde S_{\alpha_1}\cdots \tilde S_{\alpha_k}$ for
$\alpha=(\alpha_1, \ldots, \alpha_k)\in (J^\circ \cup J)^k$.
We also define the vectors $\xi_\alpha := T_\alpha 1\in H_1^{\otimes |\alpha |}$.
It is a standard fact that the formula
\begin{align*}
\tilde E:& \, \tilde \CT_P \to \tilde \CT_P^\IT\\
\tilde E(x)&:=\int_0^1 \sigma_\beta(x)d\beta 
\end{align*}
yields a faithful conditional expectation onto $\tilde \CT_P^\IT$.
Since
 \begin{align*}
\tilde E(\tilde S_\alpha \tilde S_\beta^*)&= \tilde S_\alpha \tilde S_\beta^* \int_0^1 e^{2\pi i \beta (|\alpha|-|\beta|)}d\beta = \delta_{|\alpha|, |\beta|}\tilde S_\alpha \tilde S_\beta^*\,,
\end{align*}
it follows that $\tilde E(\tilde \CT_P)={\rm span}\{\tilde S_\alpha \tilde S_\beta^*\, | \, |\alpha|=|\beta|\}$. Note that $f \tilde S_\alpha \tilde S_\beta^*=\tilde S_\alpha \tilde S_\beta^* f$ 
for $|\alpha|=|\beta|$ and 
 $f\in C(\IZ\cup \{\infty \})$.

For any irreducible representation $\psi: \CT_P\to \CB(K)$, where  $\psi(e_k)\neq 0$ for all $k$,
the gauge automorphisms can also be defined on $\psi(\CT_P)$. 
In fact, for any $v\in K_j:=\psi(e_j)K$ define the  operator $U_\beta v:=e^{2\pi i \beta j} v$. We have 
\begin{align}\label{gauge_implemented}
U_\beta \psi(S_j)U_\beta^*v &=e^{-2\pi i \beta j} U_\beta (\psi(S_j)v)= 
e^{-2\pi i \beta j} e^{2\pi i \beta (j+1)} \,,
\psi(S_j)v\,,
\end{align}
where we used that $\psi(S_j)K_j\subseteq K_{j+1}$.
Furthermore, $U_\beta$ is a unitary operator on $K$ 
because $\oplus_{j\geq 0} \psi(e_j)K$ is the whole $K$ by irreducibility.
Therefore, we have $U_\beta \psi(S_j)U_\beta^*=e^{2\pi i \beta }\psi(S_j)$.
We denote by  $\psi(\CT_P)^\IT$ the fixed point algebra under this family of automorphisms. 
There is a faithful conditional expectation defined as
\begin{align*}
&E:  \psi(\CT_P) \to  \psi(\CT_P)^\IT\\
& E(x):=\int_0^1 \Ad(U_\beta)(x)d\beta .
\end{align*}
As
 \begin{align*}
E(\psi(S_\alpha S_\beta^*))&= \psi(S_\alpha S_\beta^*) \int_0^1 e^{2\pi i \beta (|\alpha|-|\beta|)}d\beta = \delta_{|\alpha|, |\beta|}  \psi(S_\alpha S_\beta^*)\,,
\end{align*}
we have $\psi(\CT_P)^\IT =\overline{\rm span}\{\psi( S_\alpha S_\beta^*)\, | \, |\alpha|=|\beta|\}$.

 We are now going to prove that the relations of Proposition \ref{prop-rel} actually define our C$^*$-algebra $\CT_P$. 
A key ingredient of the proof are the gauge automorphisms.

By universality, $\CT_P$ is a quotient of $\tilde \CT_P$, namely
the map 
$\pi: \tilde \CT_P\to \CT_P$ given by $\pi(\tilde S_j):=S_j$
extends to a $*$-epimorphism of 
 $\tilde \CT_P$ onto $\CT_P$. We next show that this map is injective as well. 
 In other terms, the Toeplitz algebra $\CT_P$ 
 is actually a universal $C^*$-algebra.
\begin{theorem}\label{theo-faithful-reps}
The map $\pi$ defined above establishes a $*$-isomorphism between  $\tilde \CT_P$ and $\CT_P$.  
\end{theorem}

 \begin{proof} 
 All we need to show is that $\pi$ is injective.  As customary, it suffices to check injectivity on positive elements.
 For the sake of clarity, we divide the proof into several steps.

 \noindent
 \textbf{Step 1}:  $\pi$ is injective if and only if its restriction to $\tilde \CT_P^\IT$ is.\\
 Note that 
 \begin{align*}
 \pi(\tilde E(\tilde S_\alpha \tilde S_\beta^*))&= \pi(\tilde S_\alpha \tilde S_\beta^*) \int_0^1 e^{2\pi i \beta (|\alpha|-|\beta|)}d\beta \\
 &= \delta_{|\alpha|, |\beta|}  S_\alpha    S_\beta^* = S_\alpha S_\beta^* \int_0^1 e^{2\pi i \beta (|\alpha|-|\beta|)}d\beta =E(S_\alpha S_\beta^*) =E(\pi(S_\alpha S_\beta^*))\,.
\end{align*}
In particular, we have $\pi\circ \tilde E = E\circ \pi$ 
and $\pi(\tilde \CT_P^\IT)=\CT_P^\IT$.
Now the claim holds because the conditional expectation $E$ is faithful.

 \noindent
 \textbf{Step 2}: One has $e_k \tilde S_\alpha =0$ for all $|\alpha|\geq k+1$, with $k\geq 0$. \\ Indeed, for $k=0$ 
 we have $0=e_0(\sum_{i=1}^{n-1} \tilde  S_i\tilde  S_{i}^*  )e_0=\sum_{i=1}^{n-1}  e_0\tilde S_i \tilde S_{i}^*e_0 $. Since $e_0\tilde S_i\tilde  S_{i}^*e_0$ is a positive element, this means that $e_0\tilde S_i \tilde S_{i}^*e_0=0$ for all $i$.
 This implies that $0=\|e_0\tilde S_i \tilde S_{i}^*e_0\|=\| e_0 \tilde S_i\|^2$, that is $e_0 \tilde S_i=0$.
  Now, for $k\geq 1$, 
 an inductive argument 
 can be used 
 thanks to
 $e_k  \tilde S_\alpha =  \tilde S_{\alpha_1} e_{k-1}( \tilde S_{\alpha_2}\cdots \tilde S_{\alpha_h})$.
 
  \noindent
 \textbf{Step 3}: Set $\tilde \CT_P^{\IT,k}$ as the linear span of $e_k \tilde S_\alpha \tilde S_\beta^* = e_k \tilde S_\alpha \tilde S_\beta^* e_k$, where $|\alpha|=|\beta|= k$.  
 Then $\tilde \CT_P^{\IT,k}$ is an algebra isomorphic to $M_{\dim H_k}(\IC)$.\\
 Indeed, the vector space $\tilde \CT_P^{\IT,k}$ is an algebra because of 
 \eqref{eq4rel}, \eqref{eq5rel}, \eqref{eq6rel}, Step 2, and the equality $e_k \tilde S_\alpha \tilde S_\beta^* = \tilde S_\alpha e_0\tilde S_\beta^*= \tilde S_\alpha \tilde S_\beta^*e_k$. 
We will show that  both $\dim \CT_P^{\IT,k}$ and $\dim \tilde \CT_P^{\IT,k}$ are equal to $(\dim H_k)^2$. 
Since $\pi(\tilde \CT_P^{\IT,k})=\CT_P^{\IT,k}$, it will follow
 that the restriction of $\pi$ to $\tilde \CT_P^{\IT,k}$
 is an isomorphism.\\
Clearly,  $\CT_P^{\IT,k}$ is a subalgebra of $\CB(H_k)$.
 The vectors $g_k \xi_\alpha$, with $|\alpha |=k$, span the vector space $H_k$.
Choose $\dim H_k$ linearly independent vectors among them, say $\{g_k \xi_{\alpha(i)}\}_{i=1}^{\dim H_k}$.
It is easy to see that
$$
S_{\alpha(i)} S_{\alpha(j)}^*  g_k  \xi_{\alpha(l)}=\langle  g_k   \xi_{\alpha(l)}, \xi_{\alpha(j)} \rangle  g_k   \xi_{\alpha(i)}\,.
$$
We now show that $\{S_{\alpha(i)} S_{\alpha(j)}^*\}_{i,j=1}^{\dim H_k}$ are linearly independent.
Suppose that $\sum_{i,j} c_{ij} S_{\alpha(i)} S_{\alpha(j)}^*=0$.
Then, for any $l$ we have
\begin{align*}
0=\sum_{i,j=1}^{\dim H_k} c_{ij} S_{\alpha(i)} S_{\alpha(j)}^*g\xi_{\alpha(l)}&=\sum_{i,j=1}^{\dim H_k} c_{ij} \langle g_{k}\xi_{\alpha(l)}, \xi_{\alpha(j)}\rangle  g \xi_{\alpha(i)}\\
&=\sum_{i=1}^{\dim H_k}\left( \sum_{j=1}^{\dim H_k} c_{ij} \langle g_{k}\xi_{\alpha(l)}, \xi_{\alpha(j)}\rangle \right) g_k \xi_{\alpha(i)} .
\end{align*}
Since the $g_k \xi_{\alpha(i)}$'s are linearly independent, we get
\begin{align*}
0&=\sum_{j=1}^{\dim H_k} c_{ij} \langle g_{k}\xi_{\alpha(l)}, \xi_{\alpha(j)}\rangle = \langle g_{k}\xi_{\alpha(l)}, \sum_{j=1}^{\dim H_k} \bar c_{ij} g_{k}\xi_{\alpha(j)}\rangle .
\end{align*}
As $l$ is arbitrary, this means
\begin{align*}
0&=\sum_{j=1}^{\dim H_k} \bar c_{ij}g_{k}\xi_{\alpha(j)}\,,
\end{align*}
and again   the linear independence of the vectors lead to $c_{ij}=0$ for all $i, j$. This means that the elements $S_{\alpha(i)} S_{\alpha(j)}^*$ are 
all linearly independent and thus 
 $\dim \CT_P^{\IT,k}= (\dim H_k)^2$. As $\pi(\tilde \CT_P^{\IT,k})= \CT_P^{\IT,k}$, it follows that $\dim \tilde  \CT_P^{\IT,k}\geq (\dim H_k)^2$.
 
 We now show that $\dim \tilde \CT_P^{\IT,k}= (\dim H_k)^2$.
 Consider the unital free abstract algebra generated by $s_1$, \ldots, $s_{n-1}$ (without their adjoints), 
quotient it by the ideal generated by
 $$
 y:=\sum_{b=1}^{2r-1} a_1 \langle 1,-b\rangle     s_{j_b}  s_{j_{b}}    + \sum_{j=r+1}^{n-r} a_j   s_j   s_{\bar j}\,.
 $$
 Denote this algebra by $\CC$  and by $\CC_k$ the vector space linearly generated by the product of exactly $k$ generators.
We will prove that the following relation holds
$$
\dim \CC_{k+1}=\dim \CC_k \cdot \dim \CC_1 - \dim \CC_{k-1} .
$$
Note that $\dim \CC_0=\dim H_0$ and $\dim \CC_1=\dim H_1$.
Consider the map $\varphi_k : \CC_k \otimes \CC_1\to \CC_{k+1}$ defined as $\varphi_k(x\otimes s_i):= xs_i$ for all $x\in \CC_k$ and all $i$'s.
Now the map is surjective and its kernel is $\CC_{k-1} \tilde y$, where
$$
 \tilde y:=\sum_{b=1}^{2r-1} a_1 \langle 1,-b\rangle     s_{j_b}  \otimes s_{j_{b}}    + \sum_{j=r+1}^{n-r} a_j   s_j   \otimes  s_{\bar j}\,.
 $$ 
 As the $s_j$'s are linearly independent, the map $\CC_{k-1}\ni z\mapsto z\otimes \tilde y\in \CC_{k}\otimes \CC_1$ is injective. Otherwise one would have $\CC_k=\sum_{j\in J\cup J^\circ} \CC_{k-1} s_j=0$.
This shows that $\dim \CC_k =\dim H_k$.

Set   $\CA_k:= {\rm span}\{\tilde S_\alpha \, | \, |\alpha |=k\}\subset \tilde \CT_P$.
There exists an algebra morphism 
$p: \CC \to \tilde \CT_P$
defined by $p(s_j):=\tilde S_j$.
Since $\pi(\CC_k)=\CA_k$, 
$\dim \CA_k$
must be smaller than or equal to $\dim \CC_k$.
Therefore, $\dim \CA_k\leq \dim H_k$.  From the inequality
$\dim \tilde  \CT_P^{\IT,k}\leq (\dim \CA_k)^2$
 it follows  $\dim    \CA_k \geq \dim H_k$ for all $k\geq 0$, otherwise the inequality $\dim \tilde  \CT_P^{\IT,k}\geq (\dim H_k)^2$ would not be true.  Therefore, we have
$\dim \CA_k =\dim H_k$.

 \noindent
 \textbf{Step 4}: We want to show that the elements of the form  $e_k \tilde S_\alpha \tilde S_\beta^*$ with $|\alpha| =|\beta|\leq k-1$ are in the linear span of those with
$|\alpha|=|\beta|=k$.\\
 Indeed, for $|\alpha|=|\beta|<k$ we have
 \begin{align*}
  &\sum_{i\in J^\circ \cup J} e_k \tilde S_\alpha   \tilde S_i \tilde S_{i}^*\tilde S_\beta^* = e_k \tilde S_\alpha \left(  \sum_{i\in J^\circ \cup J}   \tilde S_i \tilde S_{i}^*\right)\tilde S_\beta^* \\ 
 &=  e_k \tilde S_\alpha   (1-e_0)\tilde S_\beta^* =  \tilde S_\alpha e_{k-|\alpha|}  (1-e_0)\tilde S_\beta^* \\
 &=  \tilde S_\alpha e_{k-|\alpha|}  \tilde S_\beta^* =  e_k \tilde S_\alpha \tilde S_\beta^* \,.
 \end{align*}

 \noindent
 \textbf{Step 5}:  The restriction of  $\pi$ to each $\tilde \CT_P^{\IT,k}$ is injective.\\ 
 This follows from the simplicity of $\tilde \CT_P^{\IT,k}$ 
 because the restriction of $\pi$ to $\tilde \CT_P^{\IT,k}$ is non-zero.

 \noindent
 \textbf{Step 6}: One has that
   $\overline{\rm span}_{k\geq 0}\tilde \CT_P^{\IT,k}  \cong \oplus_{k\geq 0} \tilde \CT_P^{\IT,k}$ 
and   
   $\overline{\rm span}_{k\geq 0}  \CT_P^{\IT,k}  \cong \oplus_{k\geq 0}  \CT_P^{\IT,k}$ 
   are  ideals of  $\tilde \CT_P^{\IT}$ and $\CT_P^{\IT}$, respectively. \\
   That ${\rm span}_{0\leq k \leq m}\tilde \CT_P^{\IT,k}  \cong \oplus_{0\leq k\leq m} \tilde \CT_P^{\IT,k}$ as $C^*$-algebras follows from the defining relations of the algebra, thus the conclusion is reached by a standard density argument. That $\overline{\rm span}_{k\geq 0}  \CT_P^{\IT,k}  \cong \oplus_{k\geq 0}  \CT_P^{\IT,k}$  is obvious. The final claim follows from Step 2.
 
 \noindent
 \textbf{Step 7}: The restriction of $\pi$ to $ \oplus_{k\geq 0} \tilde \CT_P^{\IT,k}$ is injective.\\
Let $x$ be a non-zero element of $\oplus_{k\geq 0} \tilde \CT_P^{\IT,k}$. Then $xe_n\neq 0$ for some $n$.
The claim now follows from Step 5.

 \noindent
 \textbf{Step 8}:  $\pi$ is faithful on $\CT^{\IT}_P$. \\
 We claim that $xe_n\neq 0$ for some $n$. 
 Take a non-zero positive element $x\in \CT_P^{\IT}$.  
Consider the functions 
$$
f_k(n)=1+\frac{1}{(1+n)(1+k)}\in C(\IZ_+\cup \{\infty \})\,.
$$
Clearly, each $f_k$ is in $ C(\IZ_+\cup \{\infty \})$.
Note that 
$f_k\to 1$ as $k$ goes to infinity because 
$$
\| f_k -1\|_\infty = \sup_n |f_k(n)-1|\leq |1-(1+k)^{-1}|\to 0\,.
$$
If $xe_n$ were to be $0$ for all $n$, 
then $xf_k$ would be $0$ for all $k$
and this is absurd because  $xf_k\to x$.
So $xe_n$ is a non-zero element of $\CT^{\IT}\cap \oplus_{k\geq 0} \tilde \CT_P^{\IT,k}$, for some $n$. 
It follows from Step 7 that $\pi(xe_n)\neq 0$ and in particular $\pi(x)\neq 0$.
 \end{proof}
 
\begin{remark}
The strategy employed in the above proof is substantially different from that of \cite{Nesh}, as it does not deduce universality from that of $\CO_P$.
Our approach is akin to that of \cite{BPRS}, where the authors studied certain graph C$^*$-algebras .
 \end{remark}
 We single out a result from the previous proof.
 \begin{proposition}
 For $k\in\IN_0$, set $\CT_P^{\IT, k}={\rm span}\{S_\alpha S_\beta^* e_n\, | \, |\alpha |=|\beta |=n \}$.
 Then $\oplus_{k\in\IN_0}\CT_P^{\IT, k}$ is an ideal of $\CT_P$ and $\CT_P^{\IT, k}\cong M_{\dim H_k}(\IC)$.
 \end{proposition}

 We now work towards proving the universality of our Cuntz-Pimsner algebra as well.
 By setting all the $e_n$ equal to zero in the presentation of $\CT_P$, 
 it is natural to introduce the following universal C$^*$-algebra. 
\begin{definition}\label{universalOP}
Let
  $\tilde \CO_P$ be the universal unital C$^*$-algebra generated by
  $\tilde S_1$, \ldots, $\tilde S_{n-1}$  satisfying the relations 
  \begin{align}
  &\sum_{i\in J^\circ \cup J} \tilde S_i \tilde S_{i}^*  = 1  \label{eq2relO}\\ 
&\sum_{b=1}^{2r-1} a_1 \langle 1,-b\rangle    \tilde S_{j_b}\tilde S_{j_{b}}    + \sum_{j=r+1}^{n-r} a_j \tilde S_j \tilde S_{\bar{j}}   =0 \label{eq3relO}\\
&   \tilde S^*_j\tilde S_i= \delta_{ij} -\varphi_\infty \bar{a}_{ i} a_j\tilde S_{\bar j}\tilde S_{\bar i}^*\qquad i, j \in J^\circ=\{r+1, \ldots, n-r\} \label{eq4relO}\\
&   \tilde S^*_j\tilde S_{j_s}= -   \varphi_\infty\bar{a}_{j_s} a_j    \langle  s, 1\rangle  \tilde S_{\bar j}\tilde S_{j_s}^* \qquad j_s\in J, j \in J^\circ
 \label{eq5relO}\\
&  \tilde S^*_{j_{s'}}\tilde S_{j_s}= \delta_{s, s'}  - q 	  \langle s-s', 1\rangle  	\tilde S_{j_s}\tilde S_{j_{s'}}^* \qquad j_{s}, j_{s'} \in J=\{j_1, \ldots, j_{2r-1}\}
			\label{eq6relO}
\end{align} 
where $\varphi_\infty=q^2+q+1$ (note that $\varphi_\infty \lambda = q$).
  \end{definition}
  Thanks to the universality of $\tilde \CO_P$, we may define the morphism $\eta: \tilde \CO_P\to \CO_P$ by setting $\eta(\tilde S_i)=S_i$ for all $i$.  
\begin{corollary}\label{universalityOP}
The morphism $\eta$ is an isomorphism.
\end{corollary}
\begin{proof}
By definition, the algebra $\CO_P$ is the quotient of $\CT_P$ by the compact operators $\CK(\CF_\CH)$.
Recall that the latter is generated by the projection $e_0$ and that  $e_n$ is compact for every $n\in \IN_0$. 
There exists a map $\theta: \CO_P\to \tilde \CO_P$, mapping $S_i$ to $\tilde S_i$ because every 
representation of $\CT_P$, where all $e_n$'s are zero, projects onto a representation of $\CO_P$.
%
\end{proof}
  
In the final part of this article, we investigate some features of  the representation theory of $\CT_P$.
\begin{lemma}\label{dichotomy}
Let $\psi: \CT_P\to \CB(K)$ be an irreducible representation, with $\dim K\geq 2$. Then
either of the following cases occurs: i) $\psi(e_n)=0$ for all $n\geq 0$; ii)   $\psi(e_n)\neq 0$ for all $n\geq 0$. Moreover, in case ii) the Hilbert space $K$ is infinite dimensional.
\end{lemma}
\begin{proof} 

First we show that one of these three cases must occur: 1) $\psi(e_n)=0$ for all $n\geq 0$; 2) $\psi(e_0)\neq 0$ and $\psi(e_n)=0$ for all $n\geq 1$; 3) 
 $\psi(e_n)\neq 0$ for all $n\geq 0$.\\
Note that the ideal generated by $e_0$ contains all $e_k$, with $k\geq 0$.
In fact, for $k\geq 0$, we have
 $$
 \sum_{i\in J^\circ\cup J}   S_ie_{k} S_i^*= e_{k+1} \left(  \sum_{i\in J^\circ\cup J}   S_i S_i^*\right)=e_{k+1} (1-e_0)=e_{k+1}\,.
 $$
 This formula also shows that $\psi(e_k)=0$,  implies $\psi(e_{k+1})=0$.
 We want to show that that $\psi(e_k)=0$ also  implies $\psi(e_{k-1})=0$, provided that $k\geq 2$. Indeed, one has
 \begin{equation}\label{reverse_equation}
 \begin{aligned}
& \sum_{i\in J^\circ\cup J} |a_{\bar{i}}|^2 S_i^* e_k  S_i =   \sum_{i\in J^\circ\cup J} |a_{\bar{i}}|^2 S_i^*   S_i e_{k-1}=  \left( \sum_{i\in J^\circ\cup J} |a_{\bar{i}}|^2 S_i^*    S_i \right) e_{k-1}\\
 &=  \left( \sum_{i\in J^\circ\cup J} |a_{\bar{i}}|^2- |a_{\bar{i}}|^2 |a_{i}|^2\varphi    S_i    S_i^* \right) e_{k-1}\\
 &=  (1-\lambda)e_{k-1}- \sum_{i\in J^\circ\cup J} \varphi(k-1) \lambda^2   S_i    S_i^* e_{k-1}\\
  &=  (1-\lambda-\lambda^2 \varphi(k-1)    ) e_{k-1}\,.
 \end{aligned}
 \end{equation}
 Recall that in \eqref{formulaphi} we saw that $\varphi(m)=\lambda^{-1}[m]_q/[m+1]_q$.
For $q\in (0,1)$,
the function $1-\lambda-\lambda^2 \varphi(m)$ can be written as function of $q$ as
\begin{align*}
&1-\lambda-\lambda^2 \varphi(m)\\
&=1-\lambda- \lambda\frac{q^m-q^{-m}}{q^{m+1}-q^{-m-1}}\\
&=\lambda\left( 	\lambda^{-1}-1- \frac{q^m-q^{-m}}{q^{m+1}-q^{-m-1}}	\right)\\
&=\lambda\left( 	\frac{(q+q^{-1})(q^{m+1}-q^{-m-1})-q^m+q^{-m}}{q^{m+1}-q^{-m-1}}	\right)\\
&=\lambda\left( 	\frac{q^{m+2}-q^{-m-2}}{q^{m+1}-q^{-m-1}}	\right)\,,
\end{align*}
 which is always different from $0$  and thus  $\psi(e_k)=0$ if and only if $\psi(e_{k-1})=0$. 
 If $q=1$, $\varphi(m)=3m(m+1)^{-1}$ and thus $\psi(e_k)=0$ if and only if $\psi(e_{k-1})=0$

Denote by $e$ the strong limit of  $\sum_{0\leq k\leq n} \psi (e_k)$.
The projection $e$ commutes with $\psi( \CT_P)$ and, 
as $\psi$ is irreducible, then either
  $e$ is  $0$ or  $1$. The former case occurs if and only if $\psi(e_k)=0$ for all $k\geq 0$.
 
 In Case 2) we have 
$\psi(e_0)=1$ and $\psi(e_k)=0$ for all $k\geq 1$.
In particular, $\psi( S_i)=0$ because $0=\psi(e_1)\psi( S_i)=\psi(S_i)\psi(e_{0})=\psi( S_i)$. Therefore, only the cases 1) and 3) can occur. 
Finally, set $K_i:=\psi(e_i)K$. 
  If all $\psi(e_k)\neq 0$, it follows that $\dim K=\infty$.
\end{proof}
  
 The next result shows that the $e_0$ generates an essential ideal. 
\begin{proposition}\label{ideal_structure} 
Any non-trivial ideal of $\CT_P$  contains
$\langle e_0\rangle =\CK(\CF_\CH)$. 
\end{proposition}
\begin{proof}
Let $J$ be a non-trivial ideal of $\CT_P$ and
take a non-zero element $x$ in it. 
 By the irreducibility of $\CT_P$
in $\CB(\CH)$,
there must exist a non-negative integer $m$ 
such that  
$e_mx$ 
is different from zero.
Since $e_m$ is a finite-dimensional projection, 
the ideal generated by $e_mx$ coincides with $\CK(\CF_\CH)$. 
\end{proof}
 
We end this paper putting everything together with a refinement of  Lemma \ref{dichotomy}.
\begin{corollary} \label{dichotomy-improved}
Let $\psi: \CT_P\to \CB(K)$ be an irreducible representation, with $\dim K\geq 2$. Then
either of the following cases occurs
\begin{enumerate}
\item $\psi$
factorizes through a representation of $\CO_P$;
\item   $\psi(e_m)\neq 0$ for all $m\geq 0$ and the representation is faithful.
\end{enumerate}
\end{corollary}
\begin{proof}
We already know from  Lemma \ref{dichotomy}
that either  $\psi(e_m)=0$ for all $m\geq 0$, or   $\psi(e_m)\neq 0$ for all $m\geq 0$. 
The first statement follows from the fact that 
$e_0$ generates $\CK(\CF_\CH)$. 
As for the second statement, we need to show that the 
representation is faithful.
As observed in \eqref{gauge_implemented}, the gauge automorphisms are unitarily implemented in Case (ii).
Now, by Proposition \ref{ideal_structure} any non-zero, 
 element in the kernel of $\psi$ 
  generates an ideal containing 
$e_0$
and 
thus we cannot have $\psi(e_m)\neq 0$ for all $m\geq 0$ with $\psi$ not being faithful.
\end{proof}

A C$^*$-algebra somehow related to those studied in this article, and yet generalising Cuntz-like algebras   \cite{Cuntz, LarsenLi, ACR4, ACRsurvey}, are the 
Pythagorean C$^*$-algebra $P_k$. These are the universal C$^*$-algebras generated by $k$ elements $A_1$, \ldots, $A_k$ such that $\sum_{i=1}^k A_i^*A_i=1$, \cite{BJ}.
Any representation of these C$^*$-algebras allows one to construct a representation of the Thompson group $V_k$.
There is a natural projection from $P_k$ to $\CO_P$, whenever $\dim H_1=k$, and more generally from any standard subproduct system there is a projection from $P_{\dim H_1}$,  to the corresponding Cuntz-Pimsner algebra. It would be of interest to investigate the representations of the Thompson groups arising from the Motzkin  and other subproduct systems.

\section*{Acknowledgements}
 V.A. would like to thank Levi Ryffel for teaching him how to efficiently draw the Motzkin diagrams in LaTeX. 
 V.A. is partially supported by Sapienza Universit\`a di Roma (Progetto di Ateneo Dipartimentale 2024 “New research trends in Mathematics at Castelnuovo”). \\
The authors are members of INDAM-GNAMPA.\\
S.D.V. and S.R. acknowledge the support of INdAM-GNAMPA Project 2024 ``Probabilit\`a quantistica e applicazioni'',
  by Italian PNRR Partenariato Esteso PE4, NQSTI, and Centro Nazionale CN00000013 CUP H93C22000450007.

\end{document}